\documentclass[a4paper,11pt]{article}

\usepackage[margin=1in]{geometry}
\usepackage{amsmath,amssymb,amsthm,mathtools,mathrsfs,bm}
\usepackage{enumitem}
\usepackage{microtype}
\usepackage[pagebackref=true, colorlinks, linkcolor=blue, anchorcolor=red,citecolor=blue]{hyperref}

\numberwithin{equation}{section}

\newtheorem{theorem}{Theorem}[section]
\newtheorem{proposition}[theorem]{Proposition}
\newtheorem{lemma}[theorem]{Lemma}
\newtheorem{corollary}[theorem]{Corollary}

\theoremstyle{remark}
\newtheorem{remark}[theorem]{Remark}

\newcommand{\R}{\mathbb{R}}

\newcommand{\Sd}{\mathbb{S}^{d-1}}
\newcommand{\cP}{\mathcal{P}}
\newcommand{\cL}{\mathcal{L}}
\newcommand{\cD}{\mathcal{D}}
\newcommand{\cR}{\mathcal{R}}
\newcommand{\cA}{\mathcal{A}}

\newcommand{\Var}{\operatorname{Var}}
\newcommand{\Cov}{\operatorname{Cov}}
\newcommand{\Tr}{\operatorname{Tr}}
\newcommand{\Lip}{\operatorname{Lip}}
\newcommand{\dist}{\operatorname{dist}}
\newcommand{\esssup}{\operatorname*{ess\,sup}}

\newcommand{\dd}{\,\mathrm{d}}
\newcommand{\Law}{\mathrm{Law}}
\newcommand{\Wick}[1]{\mathopen{:}#1\mathclose{:}}

\begin{document}

\title{\textbf{Rigidity and Quantitative Stability of the Sliced Wasserstein Deficit}}
\author{Bang-Xian Han\thanks{School of Mathematics, Shandong University, Jinan, 250100, China. Email: hanbx@sdu.edu.cn.}}
\date{\today}

\maketitle

\begin{abstract}
The sliced Wasserstein distance $SW_2(\mu,\nu)$ compares high-dimensional probability measures by averaging one-dimensional optimal transport distances over linear projections. Although sliced Wasserstein distances are now standard computational tools in statistics, imaging, and machine learning, the rigidity and stability questions behind the elementary comparison
\[
 SW_2^2(\mu,\nu)\leq \frac1d W_2^2(\mu,\nu)
\]
have not been systematically studied.

For $\mu,\nu\in\cP_2(\R^d)$ with $\mu\ll\cL^d$,   $d\ge2$,  and the sliced Wasserstein deficit 
\[
 \cD(\mu,\nu):=\frac1d W_2^2(\mu,\nu)-SW_2^2(\mu,\nu),
\]
we prove that $\cD(\mu,\nu)=0$ if and only if the Brenier map $T=\nabla\varphi$ from $\mu$ to $\nu$ is homothetic affine. We also prove an analogous rigidity theorem for general Grassmannian projections.

For quantitative stability, we introduce the sliced Poincar\'e--Korn (SPK) constant. Under a positive SPK spectral gap, we prove a stability estimate for the sliced Wasserstein deficit, up to a one-dimensional Lipschitz scale for the projected monotone transports. For Gaussian measures we compute the sharp SPK constant.
We also show by an example that anisotropic Gaussians give an obstruction: neither a Bakry--\'Emery lower curvature bound nor a usual Poincar\'e inequality alone is sufficient for a global sliced Poincar\'e--Korn inequality.
\end{abstract}



\tableofcontents
\section{Introduction}
\label{sec:intro}
\subsection{Background}
For probability measures $\mu,\nu\in\cP_2(\R^d)$, their quadratic Wasserstein distance (see, e.g., \cite{Santambrogio2015}) is defined by
\[
 W_2^2(\mu,\nu)
 :=\inf_{\gamma\in\Pi(\mu,\nu)}
 \int_{\R^d\times\R^d}|x-y|^2\,\dd\gamma(x,y)
\]
where $\Pi(\mu,\nu)$ denotes the set of transport plans, which are probability measures on $\R^d\times \R^d$ with marginals $\mu$ and $\nu$. For $\theta\in\Sd$, write
\[
 \pi_\theta(x)=\theta\cdot x,
 \qquad
 \mu_\theta=(\pi_\theta)_\#\mu.
\]
The sliced Wasserstein distance (see \cite{BonneelRabinPeyrePfister2015}) is defined via
\[
 SW_2^2(\mu,\nu)
 :=\int_{\Sd} W_2^2(\mu_\theta,\nu_\theta)\,\dd\sigma(\theta),
\]
where $\sigma$ is the normalized surface measure on $\Sd$. 

Note that
\[
 \int_{\Sd}\theta\otimes\theta\,\dd\sigma(\theta)=\frac1d I_d,
 \qquad
 \int_{\Sd}|\theta\cdot z|^2\,\dd\sigma(\theta)=\frac1d|z|^2.
\]
For every coupling $\gamma\in\Pi(\mu,\nu)$, $(\pi_\theta,\pi_\theta)_\#\gamma$ is also a coupling in $\Pi(\mu_\theta,\nu_\theta)$, so
\[
 W_2^2(\mu_\theta,\nu_\theta)
 \leq
 \int |\theta\cdot(x-y)|^2\,\dd\gamma(x,y).
\]
Averaging in $\theta$ and optimizing over $\gamma$ gives the elementary inequality
\begin{equation*}
\label{eq:deficit-def}
 \cD(\mu,\nu)
 :=
 \frac1d W_2^2(\mu,\nu)-SW_2^2(\mu,\nu)
 \geq0, \tag{SWD}
\end{equation*}
and we call $\cD(\mu,\nu)$ the \emph{sliced Wasserstein deficit}.

If $\mu\ll\cL^d$ and $T=\nabla\varphi$ is the Brenier map from $\mu$ to $\nu$, then
\begin{equation}
\label{eq:deficit-integral}
 \cD(\mu,\nu)
 =
 \int_{\Sd}
 \left[
 \int_{\R^d}|\theta\cdot(T(x)-x)|^2\,\dd\mu(x)
 -W_2^2(\mu_\theta,\nu_\theta)
 \right]
 \dd\sigma(\theta).
\end{equation}
Thus $\cD$ is an average of non-negative one-dimensional optimality gaps.

\medskip
Sliced Wasserstein distances were introduced as computationally tractable
variants of classical Wasserstein distances, by reducing the comparison of high-dimensional
measures to one-dimensional projections. Early developments include the work
of Rabin--Peyr\'e--Delon--Bernot on barycenters and texture mixing
\cite{RabinPeyreDelonBernot2011}; Bonnotte's thesis \cite{Bonnotte2013},
where he proved that $SW_p$ is a metric equivalent to $W_p$ on compactly
supported measures and initiated the study of gradient flows of
\(\mu\mapsto\frac12SW_2^2(\mu,\sigma)\); and the Radon/sliced barycenter
framework of Bonneel--Rabin--Peyr\'e--Pfister
\cite{BonneelRabinPeyrePfister2015}.

\medskip

Subsequent work includes Bayraktar and Guo~\cite{BayraktarGuo2021} about equivalences between Wasserstein-type
metrics and max-sliced Wasserstein distances.
Park and Slep\v{c}ev~\cite{Parkslepcev2023} proved that the sliced Wasserstein
space is complete, and characterized its tangent space. Kitagawa and Takatsu~\cite{kitagawa2024}
introduced a broader family of sliced Monge--Kantorovich metrics,
proved their completeness, separability, and duality, and showed
that these spaces are not geodesic and not bi-Lipschitz equivalent
to the classical Wasserstein space. On the gradient-flow side,
Cozzi and Santambrogio~\cite{cozzi2024} proved
long-time asymptotics for the sliced-Wasserstein flow, including
convergence when the target is Gaussian and showing that the flow
map does not converge to the optimal transport map in general.
More recently, Carlier, Figalli, M\'erigot and Wang~\cite{carlier2025} proved sharp
quantitative comparisons between $SW_1$ and $W_1$ via Radon-transform
techniques, and determined the sharp exponent for the metric
equivalence in any dimension.

\medskip

However, we are not aware of a systematic study of the equality and stability
 in this elementary comparison \eqref{eq:deficit-def}. We
ask whether equality, or approximate equality, in the averaged
one-dimensional optimality gaps forces the high-dimensional Brenier map to be
rigid. This paper proves that zero deficit characterizes homothetic affine
Brenier maps, introduces the SPK inequality needed for stability, and
computes the sharp Gaussian constant.

\medskip

This question is also related to recent work on Poincar\'e--Korn
inequalities. Carrapatoso,
Dolbeault, H\'erau, Mischler and Mouhot \cite{CarrapatosoDolbeaultHerauMischlerMouhot2022} introduced weighted Korn and
Poincar\'e--Korn inequalities, motivated in part by
hypocoercivity and kinetic theory.  Courtade and Fathi \cite{CourtadeFathi2025PK} proved
Gaussian optimality and stability for the Poincar\'e--Korn constant under
moment constraints, showing that near-optimal measures are quantitatively
close to the standard Gaussian. The SPK inequality \eqref{eq:spk-condition}
has a different form: its defect is not the symmetric gradient, but an
averaged ridge-projection conditional variance over one-dimensional
projections. Its null space is the homothetic affine family
\(x\mapsto \lambda x+b\).

\subsection{Main results}

The first theorem identifies equality cases in \eqref{eq:deficit-def}.

\begin{theorem}[Rigidity]
\label{thm:rigidity-main}
Let $d\ge2$, let $\mu,\nu\in\cP_2(\R^d)$, assume $\mu\ll\cL^d$, and let $T=\nabla\varphi$ be the Brenier map from $\mu$ to $\nu$. Then
\[
 \cD(\mu,\nu)=0
 \quad\Longleftrightarrow\quad
 T(x)=\lambda x+b\quad\mu\text{-a.e.}
\]
for some $\lambda\ge0$ and $b\in\R^d$.
\end{theorem}

The proof is given in Section~\ref{sec:rigidity}. The argument does not assume smoothness of $T$ and avoids differentiating the Brenier potential. For a random variable $X \sim \mu$, zero deficit says that the projected coupling
 $(\theta\cdot X,\theta\cdot T(X))$
is one-dimensionally optimal for almost every direction. One-dimensional cyclical monotonicity then implies directional monotonicity for almost every $\theta$. A  geometric lemma forces every chord $T(x)-T(y)$ to be a non-negative multiple of $x-y$, and a three-point argument makes this multiple constant.

\medskip

For stability we introduce the ridge defect
\begin{equation}
\label{eq:ridge-def-intro}
 \cR_\mu(u)
 :=
 \int_{\Sd}
 \mathbb E_\mu\bigl[\Var(\theta\cdot u(X)\mid \theta\cdot X)\bigr]
 \,\dd\sigma(\theta)
\end{equation}
for vector fields $u\in L^2(\mu;\R^d)$. Equivalently,
\begin{equation}
\label{eq:ridge-inf-intro}
 \cR_\mu(u)
 =
 \int_{\Sd}
 \inf_{h\in L^2(\mu_\theta)}
 \int_{\R^d}|\theta\cdot u(x)-h(\theta\cdot x)|^2\,\dd\mu(x)
 \dd\sigma(\theta).
\end{equation}

Let $\cA_d$ be the rigidity class from Theorem~\ref{thm:rigidity-main}:
\[
 \cA_d:=\{x\mapsto \lambda x+b:\lambda\in\R,
 \ b\in\R^d\}.
\]
Here and below, following Otto's calculus, define the tangent space at $\mu\in \cP_2(\R^d)$ by
\[
 {\rm Tan}_\mu:=\overline{\{\nabla \varphi: \varphi\in C_c^\infty(\R^d)\}}^{L^2(\mu;\R^d)}.
\]
The \emph{sliced Poincar\'e--Korn constant} is
\begin{equation}
\label{eq:sk-constant-intro}
 \kappa_{\mathrm{SPK}}(\mu)
 :=
 \inf_{u\in {\rm Tan}_\mu}
 \frac{\cR_\mu(u)}{\dist_{L^2(\mu)}^2(u,\cA_d)},
\end{equation}
where terms with zero denominator are omitted. 
This quantity  \(\kappa_{\rm SPK}(\mu)\) measures the failure of a
vector field to look one-dimensional after projection. Its kernel consists
precisely of the fields whose projected components depend only on the
corresponding projected variable. On tangent fields this kernel reduces to
homothetic affine maps, so the SPK inequality is the spectral-gap form of
the rigidity statement.

\medskip

The quantitative estimate under the SPK hypothesis is stated with an explicit
tangent-space defect.  The only  regularity scale
entering the estimate is \emph{projected Lipschitz scale}
\[
\Lambda=\esssup_{\theta\in\Sd}\Lip(\tau_\theta).
\]

\begin{theorem}[Stability under SPK with tangent defect]
\label{thm:spk-main}
Let $\mu\ll\cL^d$, let $T=\nabla\varphi$ be the Brenier map from $\mu$ to $\nu$, and for $\sigma$-a.e. $\theta$ let $\tau_\theta$ be the monotone transport map from $\mu_\theta$ to $\nu_\theta$. 

If $\kappa_{\mathrm{SPK}}(\mu)>0$ and $\Lambda<\infty$, then
\begin{equation}
\label{eq:spk-main-stab}
 \dist_{L^2(\mu)}(T,\cA_d)
 \leq
 \sqrt{\frac{\Lambda}{\kappa_{\mathrm{SPK}}(\mu)}}\,
 \cD(\mu,\nu)^{1/2}
 +
 \left(1+\frac1{\sqrt{d\,\kappa_{\mathrm{SPK}}(\mu)}}\right)
 \dist_{L^2(\mu)}(T,{\rm Tan}_\mu).
\end{equation}
In particular, if $T\in {\rm Tan}_\mu$, then
\begin{equation}
\label{eq:spk-main-stab-tangent}
 \dist_{L^2(\mu)}^2(T,\cA_d)
 \leq
 \frac{\Lambda}{\kappa_{\mathrm{SPK}}(\mu)}\,
 \cD(\mu,\nu).
\end{equation}
\end{theorem}

The proof uses a one-dimensional Fenchel-gap inequality in Section~\ref{sec:spk}. The factor $\Lambda$ is a scale parameter and cannot be removed in this form: replacing a non-affine Brenier map by a large multiple multiplies the distance to $\cA_d$ quadratically, while the one-dimensional Fenchel gaps scale only linearly.

The last term in \eqref{eq:spk-main-stab} is a Sobolev-approximation defect:
it measures the distance from the Brenier map to the Otto tangent space.
It vanishes in standard settings where compactly supported smooth gradients
are dense in the corresponding weighted Sobolev gradient norm; see
Lemma~\ref{lem:tan-density-criterion}.

\paragraph*{The Gaussian as the model case.}
The standard Gaussian is the model case for the SPK inequality. Its Hermite
decomposition gives a direct spectral computation and leads to the following
constant.

\begin{theorem}[Sharp Gaussian sliced Poincar\'e--Korn inequality]
\label{thm:gaussian-sk-intro}
For the standard Gaussian measure $\gamma_d=N(0,I_d)$, $d\ge2$,
\begin{equation}
\label{eq:gaussian-sk-intro}
 \cR_{\gamma_d}(u)
 \geq
 \frac{d-1}{d(d+2)}
 \dist_{L^2(\gamma_d)}^2(u,\cA_d)
\end{equation}
for every gradient field $u\in {\rm Tan}_{\gamma_d}$. The coefficient is sharp; equivalently,
\[
 \kappa_{\mathrm{SPK}}(\gamma_d)=\frac{d-1}{d(d+2)},
 \qquad
 \overline\kappa_{\mathrm{SPK}}(\gamma_d)
 =
 \frac{d-1}{d+2}
\]
where the normalized constant is defined as
$
 \overline\kappa_{\mathrm{SPK}}:=d\,\kappa_{\mathrm{SPK}}$.
\end{theorem}

Combining Theorem~\ref{thm:gaussian-sk-intro} with Theorem~\ref{thm:spk-main} yields Gaussian stability:
\begin{equation}
\label{eq:gaussian-stability-intro}
 \dist_{L^2(\gamma_d)}^2(T,\cA_d)
 \leq
 \frac{d(d+2)}{d-1}\,\Lambda\,
 \cD(\gamma_d,\nu).
\end{equation}
The proof of Theorem~\ref{thm:gaussian-sk-intro} is based on the Hermite decomposition and a Funk--Hecke tensor estimate. The affine trace-free modes have a larger ratio, while a third-order trace component attains the sharp global SPK constant; see Section~\ref{sec:gaussian}.

\medskip

Finally, we show that isotropic normalization is essential.

\begin{proposition}[Anisotropic Gaussian obstruction]
\label{prop:aniso-obstruction-intro}
There exists a family of centered Gaussian measures $\mu_\varepsilon$ on $\R^2$ satisfying a uniform Bakry--\'Emery condition ${\rm BE}(1,\infty)$ and a uniform Poincar\'e inequality, but
\[
 \kappa_{\mathrm{SPK}}(\mu_\varepsilon)\to0
 \qquad~~\text{as}~~ \varepsilon\downarrow0.
\]
Consequently no lower bound for $\kappa_{\mathrm{SPK}}$ can depend only on a Bakry--\'Emery lower curvature bound or on the usual Poincar\'e constant.
\end{proposition}

\medskip

\paragraph{Further research}
The anisotropic Gaussian example suggests that isotropic normalization is
essential. This leaves the question whether isotropic log-concave measures
satisfy a dimension-free SPK bound, in analogy with the KLS \cite{KLS1995}
and slicing problems \cite{KlartagLehec2025}.

The question is not reducible to the KLS or slicing problems. Possible first
cases are unconditional measures, uniformly log-concave measures with
covariance normalization, and heat-flow regularizations. Eldan's stochastic
localization \cite{Eldan2013} and Ciosmak's leaf decomposition
\cite{CiosmakLeaves} may  be relevant tools.

\paragraph{Organization}

Section~\ref{sec:rigidity} proves the zero-deficit rigidity theorems, including a higher-dimensional generalization. Section~\ref{sec:spk} introduces the sliced Poincar\'e--Korn constant and proves stability under the SPK hypothesis with an explicit tangent-space defect, including a discussion of the Lipschitz parameter $\Lambda$. Section~\ref{sec:gaussian} proves the sharp Gaussian SPK inequality and Gaussian stability. Section~\ref{example} gives further examples: bounded Gaussian perturbations, compact classes for fixed measures, and the anisotropic Gaussian counterexample. The appendices contain the linear algebra computations, mainly from the representation theory of the orthogonal group on symmetric spaces.

\section{Rigidity of the Sliced Wasserstein Deficit}
\label{sec:rigidity}

In this section we prove Theorem~\ref{thm:rigidity-main}. The proof has three elementary steps. Zero deficit gives one-dimensional
monotonicity in almost every direction. This forces every chord of \(T\) to be
parallel to the corresponding chord in the source. Finally, a three-point
argument makes the proportionality factor constant.
In Theorem \ref{thm:grassmannian-rigidity} we use a different argument for
Grassmannian projections.

\subsection{Deficit decomposition}

Let $\mu\ll\cL^d$ and let $T=\nabla\varphi$ be the Brenier map from $\mu$ to $\nu$. For every direction $\theta\in\Sd$ define
\[
 g_\theta(T)
 :=
 \int |\theta\cdot(T(x)-x)|^2\,\dd\mu(x)
 -W_2^2(\mu_\theta,\nu_\theta).
\]
Then $g_\theta(T)\ge0$ and
\begin{equation}
\label{eq:deficit-average}
 \cD(\mu,\nu)=\int_{\Sd}g_\theta(T)\,\dd\sigma(\theta).
\end{equation}
Indeed, for any $X\sim \mu$, the projected coupling $(\theta\cdot X,\theta\cdot T(X))$ is an admissible coupling between $\mu_\theta$ and $\nu_\theta$, so $g_\theta\ge0$. Averaging in $\theta$ gives
\[
 \int_{\Sd}\int |\theta\cdot(T(x)-x)|^2\,\dd\mu(x)\dd\sigma(\theta)
 =
 \frac1d\int |T(x)-x|^2\,\dd\mu(x)
 =
 \frac1d W_2^2(\mu,\nu),
\]
which proves \eqref{eq:deficit-average}.

\begin{lemma}[Directional signs force parallelism]
\label{lem:parallel}
Let $a,b\in\R^d$ with $a\ne0$. If
\[
 (\theta\cdot a)(\theta\cdot b)\ge0
\]
for $\sigma$-a.e. $\theta\in\Sd$, then $b=\lambda a$ for some $\lambda\ge0$.
\end{lemma}

\begin{proof}
If $a$ and $b$ are not collinear, the two hemispheres $\{\theta:\theta\cdot a>0\}$ and $\{\theta:\theta\cdot b<0\}$ have an intersection of positive surface measure. On this intersection the product is negative. Thus $b=\lambda a$ for some $\lambda\in\R$. The sign condition forces $\lambda\ge0$.
\end{proof}

\begin{lemma}[Constancy from three points]
\label{lem:three-points}
Let \(d\ge2\), let \(\mu \ll\mathcal L^d\) be a probability measure, and let
\(T\in L^0(\mu;\mathbb R^d)\). Suppose that for \(\mu \otimes \mu\)-a.e. pair
\((x,y)\),
\[
 T(x)-T(y)=\lambda(x,y)(x-y)
\]
for some measurable \(\lambda(x,y)\in\mathbb R\). Then
\(\lambda(x,y)=\lambda_0\) for \(\mu\otimes \mu\)-a.e. \((x,y)\), and hence
\[
 T(x)=\lambda_0 x+b
 \qquad m\text{-a.e.}
\]
for some \(b\in\mathbb R^d\).
\end{lemma}

\begin{proof}
For a.e. triple $(x,y,z)$, the vectors $x-y$ and $y-z$ are linearly independent, 
\[
 T(x)-T(y)=\lambda{(x,y)}(x-y),
 \quad
 T(y)-T(z)=\lambda{(y,z)}(y-z),
\]
and
\[
 T(x)-T(z)=\lambda{(x,z)}(x-z)=\lambda{(x, z)}(x-y)+\lambda{(x,z)}(y-z).
\]
Since also $T(x)-T(z)=(T(x)-T(y))+(T(y)-T(z))$, linear independence gives
\[
 \lambda{(x, y)}=\lambda{(y,z)}=\lambda{(x,z)}
\]
for almost every triple. Fubini then implies that $\lambda(x,y)$ is a.e. constant: for a.e. $y$, the equality $\lambda(x,y)=\lambda(y,z)$ holds for almost every $(x,z)$, so both sides must equal a number depending only on $y$; finally symmetry removes the dependence on $y$.  
\end{proof}

\begin{proof}[Proof of Theorem~\ref{thm:rigidity-main}]
Assume first $\cD(\mu,\nu)=0$. Since $g_\theta\ge0$ and $\int g_\theta\,\dd\sigma=0$, we have $g_\theta=0$ for $\sigma$-a.e. $\theta$. Hence, for almost every $\theta$, the projected coupling
$
 (\theta\cdot X,\theta\cdot T(X))
$
is an optimal coupling between $\mu_\theta$ and $\nu_\theta$, and thus
\begin{equation}
\label{eq:directional-monotone}
 (\theta\cdot(x-y))(\theta\cdot(T(x)-T(y)))\ge0
\end{equation}
for $\mu\otimes\mu$-a.e. $(x,y)$ and $\sigma$-a.e. $\theta$. By Fubini, for $\mu\otimes\mu$-a.e. $(x,y)$ the sign condition \eqref{eq:directional-monotone} holds for $\sigma$-a.e. $\theta$. Since $\mu\ll\cL^d$, $x\ne y$ for almost every pair. Lemma~\ref{lem:parallel} gives
\[
 T(x)-T(y)=\lambda(x,y)(x-y),
 \qquad
 \lambda(x,y)\ge0.
\]
Lemma~\ref{lem:three-points} yields $\lambda(x,y)=\lambda_0$ for almost every pair. Fixing one point by Fubini gives $T(x)=\lambda_0 x+b$ for $\mu$-a.e. $x$.

Conversely, suppose $T(x)=\lambda x+b$ with $\lambda\ge0$. Then for every $\theta$,
\[
 \theta\cdot T(x)=\lambda(\theta\cdot x)+\theta\cdot b
\]
is the one-dimensional monotone transport from $\mu_\theta$ to $\nu_\theta$. Therefore $g_\theta=0$ for every $\theta$ and $\cD(\mu,\nu)=0$.
\end{proof}

\begin{remark}[Relation with the ridge formulation]
Zero deficit also gives
\[
 \theta\cdot T(x)=\tau_\theta(\theta\cdot x),
 \qquad \tau_\theta:\mathbb R\to\mathbb R,
\]
for a.e. direction, whenever the one-dimensional optimal map is unique.
If \(T=\nabla\varphi\) were smooth, differentiating in directions
\(v\perp\theta\) would give
\[
 P_{\theta^\perp}D^2\varphi\,\theta=0.
\]
Section~\ref{sec:grassmannian-rigidity} makes this heuristic rigorous at the
level of Hessian measures.
\end{remark}

\subsection{Higher-dimensional rigidity}
\label{sec:grassmannian-rigidity}

The one-dimensional rigidity theorem has a higher-dimensional analogue for
Grassmannian projections. The proof is different: the ridge factorization of
projected Brenier maps forces the Hessian measure of the convex potential to
be scalar.

Let $G_{d,k}$ be the Grassmannian of $k$-dimensional linear subspaces of $\R^d$, let $\pi_{d,k}$ be its Haar probability measure, and let $P_E$ denote the orthogonal projection onto $E$. Define
\[
 SW_{2,k}^2(\mu,\nu)
 :=
 \int_{G_{d,k}}
 W_2^2((P_E)_\#\mu,(P_E)_\#\nu)\,\dd\pi_{d,k}(E).
\]
Since
\[
 \int_{G_{d,k}}|P_Ez|^2\,\dd\pi_{d,k}(E)=\frac{k}{d}|z|^2,
\]
we have
\[
 SW_{2,k}^2(\mu,\nu)\le \frac{k}{d}W_2^2(\mu,\nu),
\]
and we define the Grassmannian deficit
\begin{equation}
\label{eq:grass-deficit}
 \cD_k(\mu,\nu)
 :=
 \frac{k}{d}W_2^2(\mu,\nu)-SW_{2,k}^2(\mu,\nu).
\end{equation}
For $k=1$ this is the original deficit $\cD$.

\begin{theorem}[Grassmannian rigidity]
\label{thm:grassmannian-rigidity}
Let $1\le k\le d-1$. Let $\mu=\rho\cL^d$ be concentrated on a connected open set $\Omega\subset\R^d$, with $\rho>0$ a.e. on $\Omega$, and let $T=\nabla\varphi$ be the Brenier map from $\mu$ to $\nu$. Assume that the Brenier potential $\varphi$ has a finite convex representative on $\Omega$. If
\[
 \cD_k(\mu,\nu)=0,
\]
then
\[
 T(x)=\lambda x+b
 \qquad \mu\text{-a.e.}
\]
for some $\lambda\ge0$ and $b\in\R^d$. Conversely, every map of this form has zero $k$-dimensional deficit.
\end{theorem}

\begin{lemma}[Linear algebra on the Grassmannian]\label{lem:grassmannian_algebra}
Let \(M\) be a symmetric \(d\times d\) matrix and let \(1\le k\le d-1\).
If
\[
P_E M(I-P_E)=0
\]
for \(\pi_{d,k}\)-a.e. \(E\in G_{d,k}\), then \(M=\alpha I\) for some
\(\alpha\in\mathbb R\). Moreover,
\[
\int_{G_{d,k}}\|P_E M(I-P_E)\|_{\mathrm{HS}}^2\,\dd\pi_{d,k}(E)
=
\frac{k(d-k)}{(d-1)(d+2)}
\left\|M-\frac{\operatorname{Tr}M}{d}I\right\|_{\mathrm{HS}}^2.
\]

\end{lemma}

\begin{proof}
This is proved in Appendix~\ref{app:grassmannian}.
\end{proof}

\medskip

\begin{proof}[Proof of Theorem~\ref{thm:grassmannian-rigidity}]
Let $X\sim\mu$. For $E\in G_{d,k}$ put
\[
 g_E(T)
 :=
 \mathbb E|P_E(T(X)-X)|^2
 -
 W_2^2((P_E)_\#\mu,(P_E)_\#\nu).
\]
Then $g_E(T)\ge0$ and
\[
 \cD_k(\mu,\nu)=\int_{G_{d,k}}g_E(T)\,\dd\pi_{d,k}(E).
\]
Hence $\cD_k(\mu,\nu)=0$ implies $g_E(T)=0$ for a.e. $E$. For such $E$, the coupling
$
 (P_E X,P_E T(X))
$
is an optimal quadratic coupling between $(P_E)_\#\mu$ and $(P_E)_\#\nu$. Since $(P_E)_\#\mu$ is absolutely continuous on $E$, Brenier's theorem gives a convex function $f_E:E\to\R$ such that
\begin{equation}
\label{eq:grass-ridge-factor}
 P_E T(x)=\nabla_E f_E(P_Ex)
 \qquad \mu\text{-a.e. }x.
\end{equation}
Because $\rho>0$ a.e. on $\Omega$, this identity also holds for Lebesgue-a.e. $x\in\Omega$.

Now consider the convex potential  $\varphi$  on $\Omega$ such that $T=\nabla \varphi$.   By Alexandrov's theorem, its distributional Hessian $D^2\varphi$ is a symmetric positive semidefinite matrix-valued Radon measure on $\Omega$. For  a space $E$  satisfying \eqref{eq:grass-ridge-factor} and take $v\in E$, $w\in E^\perp$,    the scalar function
$
 v\cdot T(x)=v\cdot \nabla\varphi(x)
$
depends only on $P_Ex$. Hence its derivative in the direction $w$ vanishes in distributions:
\[
 \partial_w(v\cdot\nabla\varphi)=0.
\]
Equivalently,
\begin{equation}
\label{eq:block-hessian-measure}
 v^T D^2\varphi\,w=0
 \qquad\text{as a signed Radon measure.}
\end{equation}
Thus
\[
 P_E D^2\varphi(I-P_E)=0
\]
as a matrix-valued measure for a.e. $E$.

Let $\mathfrak m:=\Tr(D^2\varphi)$ be a Radon measure and write the polar decomposition
\[
 D^2\varphi=M \mathfrak m,
\]
where $M(x)$ is symmetric positive semidefinite and $\Tr M(x)=1$ for $\mathfrak m$-a.e. $x$. By Fubini, \eqref{eq:block-hessian-measure} implies that for $\mathfrak m$-a.e. $x$,
\[
 P_E M(x)(I-P_E)=0
\]
for $\pi_{d,k}$-a.e. $E$. Lemma~\ref{lem:grassmannian_algebra} gives $M(x)=I/d$ for $\mathfrak m$-a.e. $x$. Consequently
\[
 D^2\varphi=\omega I
\]
for the scalar Radon measure $\omega:=\mathfrak m/d$.

It remains to show that $\omega$ is a constant multiple of Lebesgue measure. In distributions,
\[
 \partial_{ij}\varphi=0\quad (i\ne j),
 \qquad
 \partial_{11}\varphi=\cdots=\partial_{dd}\varphi=\omega.
\]
Fix $j$ and choose $i\ne j$. Then in distributions, 
\[
 \partial_j\omega
 =\partial_j\partial_{ii}\varphi
 =\partial_i\partial_{ji}\varphi
 =0.
\]
Thus $\nabla\omega=0$ in $\mathcal D'(\Omega)$. Since $\Omega$ is connected, $\omega=\lambda\cL^d$ for some constant $\lambda\ge0$. Hence
\[
 D^2\varphi=\lambda I
\]
in distributions on $\Omega$, and therefore
\[
 \varphi(x)=\frac{\lambda}{2}|x|^2+b\cdot x+c
\]
on $\Omega$. Thus $T=\nabla\varphi=\lambda x+b$ $\mu$-a.e.

Conversely, if $T(x)=\lambda x+b$ with $\lambda\ge0$, then for every $E$
\[
 P_ET(x)=\lambda P_Ex+P_Eb
\]
is the Brenier map from $(P_E)_\#\mu$ to $(P_E)_\#\nu$. Hence $g_E(T)=0$ for every $E$ and $\cD_k(\mu,\nu)=0$.
\end{proof}

\begin{remark}[Assumptions in the Grassmannian rigidity theorem]
Theorem~\ref{thm:grassmannian-rigidity} uses stronger assumptions than
Theorem~\ref{thm:rigidity-main}: the source is supported on a connected open
set, the density is positive a.e., and the Brenier potential is finite on that
domain. These assumptions are necessary because we works  with
the Hessian measure \(D^2\varphi\). 
\end{remark}

\section{The SPK inequality and stability}
\label{sec:spk}
In this section we identify the kernel of the ridge
defect on gradient fields,  and show that one-dimensional Fenchel gaps control the ridge defect of
the Brenier map. The SPK inequality then turns this ridge control into
\(L^2\)-stability modulo \(\mathcal A_d\). 
\subsection{Definition and spectral interpretation}

Let $\mu\in\cP_2(\R^d)$ and let
\[
 {\rm Tan}_\mu:=\overline{\{\nabla \varphi: \varphi\in C_c^\infty(\R^d)\}}^{L^2(\mu;\R^d)}.
\]
For $u\in L^2(\mu;\R^d)$ define the ridge defect by \eqref{eq:ridge-def-intro}. The sliced Poincar\'e--Korn constant (or sliced Poincar\'e--Korn spectral gap) is defined as
\[
 \kappa_{\mathrm{SPK}}(\mu)
 :=
 \inf_{
 \substack{u\in {\rm Tan}_\mu\\
 \dist_{L^2(\mu)}(u,\cA_d)>0}}
 \frac{\cR_\mu(u)}{\dist_{L^2(\mu)}^2(u,\cA_d)}.
\]
Note that $\cA_d$ is the image of the finite-dimensional vector space $\R\times\R^d$ under the linear map $(\lambda,b)\mapsto \lambda x+b$ into $L^2(\mu;\R^d)$. So it is closed  and nearest points in $\cA_d$ exist.

Because the average over directions carries a natural factor $1/d$, the normalized quantity
\[
 \overline\kappa_{\mathrm{SPK}}(\mu)=d\,\kappa_{\mathrm{SPK}}(\mu)
\]
is the quantity expected to be dimension-free in isotropic log-concave classes. For example, the sharp Gaussian theorem, Theorem~\ref{thm:gaussian-sk-intro}, gives
\[
 \overline\kappa_{\mathrm{SPK}}(\gamma_d)
 =
 \frac{d-1}{d+2}. 
\]

Equivalently, we say that $\mu$ satisfies sliced Poincar\'e--Korn inequality $\mathrm{SPK}(\kappa)$ if
\begin{equation*}
\label{eq:spk-condition}
 \dist_{L^2(\mu)}^2(u,\cA_d)
 \leq
 \frac1\kappa\cR_\mu(u)
 \qquad
 \forall u\in {\rm Tan}_\mu. \tag{SPK}
\end{equation*}

The terminology ``sliced Poincar\'e--Korn'' is motivated by the analogy with Korn-type inequalities in elasticity and classical Poincar\'e inequality. Classical Korn inequalities control vector fields modulo rigid motions by the symmetric part of the gradient, and are basic tools in linearized elasticity; see, for instance, Ciarlet's text \cite{Ciarlet1988} and Dacorogna's direct-methods treatment \cite{Dacorogna2008}. The theorem of Friesecke--James--M\"uller \cite{FrieseckeJamesMuller2002} gives a quantitative geometric rigidity estimate modulo rotations. In our setting the null space is not the Euclidean rigid-motion space but the homothetic affine family
\[
 \cA_d=\{\lambda x+b\},
\]
and the ``strain'' is replaced by an averaged conditional-variance defect over one-dimensional projections. So we adopt the name ``Poincar\'e--Korn".

Carrapatoso--Dolbeault--H\'erau--Mischler--Mouhot
\cite{CarrapatosoDolbeaultHerauMischlerMouhot2022} introduced related
weighted Poincar\'e--Korn inequalities; Courtade and Fathi
\cite{CourtadeFathi2025PK} proved
Gaussian optimality and stability for the Poincar\'e--Korn constant under moment constraints. Our sliced Poincar\'e--Korn inequality is different in
nature: it has a different defect and a different null space, but it belongs to the same general circle of Korn-type spectral gaps for probability measures.
\bigskip

\begin{lemma}[Continuity of the ridge defect]
\label{lem:ridge-continuity}
Let $\mu\in\cP_2(\R^d)$. For $\theta\in\Sd$ set
\[
 \mathcal H_\theta^\mu
 :=
 \{h(\theta\cdot x):h\in L^2(\mu_\theta)\}
 \subset L^2(\mu),
\]
and let $\Pi_\theta^\mu$ be the orthogonal projection onto this closed subspace. Then, for every $u\in L^2(\mu;\R^d)$, the ridge defect defined in \eqref{eq:ridge-inf-intro} can be represented by
\begin{equation}
\label{eq:ridge-projection-formula}
 \cR_\mu(u)
 =
 \int_{\Sd}
 \bigl\|(I-\Pi_\theta^\mu)(\theta\cdot u)\bigr\|_{L^2(\mu)}^2
 \dd\sigma(\theta).
\end{equation}
Moreover $\cR_\mu$ is continuous on $L^2(\mu;\R^d)$:
\begin{equation}
\label{eq:ridge-continuity-bound}
 \bigl|\cR_\mu(u)-\cR_\mu(v)\bigr|
 \le
 \bigl(\|u\|_{L^2(\mu)}+\|v\|_{L^2(\mu)}\bigr)
 \|u-v\|_{L^2(\mu)} .
\end{equation}

\end{lemma}

\begin{proof}
The map $h\mapsto h(\theta\cdot x)$ is an isometry from $L^2(\mu_\theta)$ into $L^2(\mu)$, hence $\mathcal H_\theta^\mu$ is closed. The conditional expectation $\mathbb E[\theta\cdot u(X)\mid \theta\cdot X]$ is exactly the orthogonal projection $\Pi_\theta^\mu(\theta\cdot u)$, which proves \eqref{eq:ridge-projection-formula}.

For a closed subspace $H$ of a Hilbert space, the map $f\mapsto \operatorname{dist}(f,H)$ is $1$-Lipschitz. Therefore, for each $\theta$,
\[
\begin{aligned}
&\left|
 \bigl\|(I-\Pi_\theta^\mu)(\theta\cdot u)\bigr\|_2^2
 -
 \bigl\|(I-\Pi_\theta^\mu)(\theta\cdot v)\bigr\|_2^2
\right|  \\
&\qquad\le
 \bigl(\|\theta\cdot u\|_2+\|\theta\cdot v\|_2\bigr)
 \|\theta\cdot(u-v)\|_2
 \le
 \bigl(\|u\|_2+\|v\|_2\bigr)\|u-v\|_2 .
\end{aligned}
\]
Integrating in $\theta$ gives \eqref{eq:ridge-continuity-bound}. 
\end{proof}

\begin{lemma}[Lipschitz continuity of the ridge residual]
\label{lem:ridge-residual-lipschitz}
For every $u,v\in L^2(\mu;\R^d)$,
\begin{equation}
\label{eq:ridge-sqrt-triangle}
 \sqrt{\cR_\mu(u+v)}
 \leq
 \sqrt{\cR_\mu(u)}+
 \sqrt{\cR_\mu(v)}.
\end{equation}
Moreover,
\begin{equation}
\label{eq:ridge-l2-bound}
 \sqrt{\cR_\mu(v)}
 \leq
 \frac1{\sqrt d}\|v\|_{L^2(\mu)}.
\end{equation}
\end{lemma}

\begin{proof}
For each $\theta$, set
\[
 r_\theta(u):=
 \inf_{h\in L^2(\mu_\theta)}
 \|\theta\cdot u-h(\theta\cdot x)\|_{L^2(\mu)}.
\]
This is the distance from $\theta\cdot u$ to the closed linear subspace
$\mathcal H_\theta^\mu\subset L^2(\mu)$. Hence $r_\theta$ is a seminorm in
$u$ and satisfies the triangle inequality. Since
$\sqrt{\cR_\mu(u)}=\|r_\theta(u)\|_{L^2(\Sd,\sigma)}$, Minkowski's inequality
gives \eqref{eq:ridge-sqrt-triangle}. For \eqref{eq:ridge-l2-bound}, choose
$h=0$ in the definition of $r_\theta$ and use the normalized spherical identity
\[
 \int_{\Sd}|\theta\cdot v(x)|^2\,\dd\sigma(\theta)=\frac1d|v(x)|^2.
\]
After integration in $x$, this gives
$\cR_\mu(v)\le d^{-1}\|v\|_{L^2(\mu)}^2$.
\end{proof}

\begin{lemma}[Vanishing of the tangent defect]
\label{lem:tan-density-criterion}
Let \(\mu=\rho\mathcal L^d\in\mathcal P_2(\mathbb R^d)\), and let
\(T=\nabla\phi\in L^2(\mu;\mathbb R^d)\) for some convex potential \(\phi\).
If \(T\) can be approximated in \(L^2(\mu;\mathbb R^d)\) by gradients
\(\nabla\psi_j\) with \(\psi_j\in C_c^\infty(\mathbb R^d)\), then
\[
 \operatorname{dist}_{L^2(\mu)}(T,\operatorname{Tan}_\mu)=0.
\]

This condition is automatic in the following standard situations:
\begin{enumerate}
\item \(\mu\) is supported on a bounded Lipschitz domain \(\Omega\),
\(0<m\le\rho\le M<\infty\) on \(\Omega\), and
\(\phi\in W^{1,2}(\Omega)\);
\item \(\mu=\gamma_d\) and \(T=\nabla p\) for a polynomial \(p\).
\end{enumerate}
\end{lemma}

\begin{proof}
The first assertion is the definition of $\operatorname{Tan}_\mu$. Under the bounded-domain assumptions, restrictions of functions in $C_c^\infty(\R^d)$ are dense in $W^{1,2}(\Omega)$. Hence one can choose $\psi_j\in C_c^\infty(\R^d)$ with $\|\nabla\psi_j-\nabla\phi\|_{L^2(\Omega)}\to0$. Since $\rho\le M$, convergence also holds in $L^2(\mu)$.

For the Gaussian statement, let $p$ be a polynomial and choose cut-offs $\chi_R\in C_c^\infty(\R^d)$ with $\chi_R=1$ on $B_R$, $\chi_R=0$ outside $B_{2R}$, and $|\nabla\chi_R|\le C/R$. Then $\chi_R p\in C_c^\infty(\R^d)$ and
\[
 \nabla(\chi_R p)-\nabla p
 =
 (\chi_R-1)\nabla p+p\nabla\chi_R.
\]
Both terms converge to zero in $L^2(\gamma_d)$ by the Gaussian tail and the polynomial growth of $p$ and $\nabla p$. Thus $\nabla p\in\operatorname{Tan}_{\gamma_d}$.
\end{proof}

The following  lemma  identifies the null space of the ridge defect,   and explains why the quotient by $\cA_d$ appears in the
SPK inequality.   See also Lemma \ref{lem:grassmannian_algebra}.

\begin{lemma}[Kernel of the ridge defect on weak gradient fields]
\label{lem:ridge-kernel-gradient}
Let $\Omega\subset\R^d$ be connected and open, and let $\mu=\rho\cL^d$ be a probability measure with $\mu(\Omega)=1$. Assume that $\rho$ is locally bounded below on $\Omega$. Let $v\in L^2(\mu;\R^d)$ have symmetric distributional derivative on $\Omega$ (equivalently, be a local distributional gradient). If the ridge defect 
$ \cR_\mu(v)=0$,
then $v(x)=\lambda x+b$ for $\mu$-a.e. $x\in\Omega$.

\end{lemma}

\begin{proof}
Since \(\rho\) is locally bounded below, \(v\in L^1_{\rm loc}(\Omega)\).
The identity \(\cR_\mu(v)=0\) means that for \(\sigma\)-a.e. \(\theta\)
there exists \(h_\theta\in L^2_{\rm loc}\) such that
\[
 \theta\cdot v(x)=h_\theta(\theta\cdot x)
 \qquad\text{for a.e. }x\in\Omega.
\]
Hence, for every \(\eta\perp\theta\),
\[
 \partial_\eta(\theta\cdot v)=0
 \qquad\text{in }\mathcal D'(\Omega).
\]
Equivalently,
\[
 P_\theta Dv(I-P_\theta)=0
 \qquad\text{in }\mathcal D'(\Omega),
\]
where \(P_\theta=\theta\otimes\theta\). Here \(Dv\) denotes the
distributional derivative matrix of \(v\), which is symmetric by assumption.

Let \(\zeta\in C_c^\infty(\Omega)\). Pairing the preceding identity with
\(\zeta\), we obtain, for \(\sigma\)-a.e. \(\theta\),
\[
 P_\theta M_\zeta(I-P_\theta)=0,
 \qquad
 M_\zeta:=\langle Dv,\zeta\rangle .
\]
The matrix \(M_\zeta\) is symmetric. By Lemma~\ref{lem:grassmannian_algebra}
with \(k=1\), \(M_\zeta\) is a scalar multiple of the identity. Since this
holds for every test function \(\zeta\), there exists a scalar distribution
\(\omega\) such that
\[
 Dv=\omega I.
\]

Thus, in distributions,
\[
 \partial_i v_j=0 \quad (i\ne j),
 \qquad
 \partial_1 v_1=\cdots=\partial_d v_d=\omega.
\]
Fix \(j\) and choose \(i\ne j\). Then
\[
 \partial_j\omega
 =
 \partial_j\partial_i v_i
 =
 \partial_i\partial_j v_i
 =
 0.
\]
Hence \(\nabla\omega=0\) in \(\mathcal D'(\Omega)\). Since \(\Omega\) is
connected, \(\omega=\lambda\) for some constant \(\lambda\in\mathbb R\).
Therefore
\[
 D(v-\lambda x)=0
\]
in distributions, and so \(v-\lambda x=b\) a.e. on \(\Omega\). Thus
\(v(x)=\lambda x+b\).   Since \(v\) has zero ridge defect, the
one-dimensional monotonicity along directions gives \(\lambda\ge0\).
\end{proof}

\subsection{Fenchel-gap stability in one dimension}

The following lemma is used to pass from the one-dimensional deficit to the
ridge defect.

\begin{lemma}[One-dimensional Fenchel gap]
\label{lem:fenchel-gap}
Let $A,B$ be real-valued random variables in $L^2$, and let $\tau$ be the monotone transport map from the law of $A$ to the law of $B$. If $\Lip(\tau)\le\Lambda$, then
\begin{equation}
\label{eq:fenchel-gap}
 \mathbb E|B-\tau(A)|^2
 \leq
 \Lambda\left(
 \mathbb E|A-B|^2-W_2^2(\Law(A),\Law(B))
 \right).
\end{equation}
\end{lemma}

\begin{proof}
If $\Lambda=0$, then $\tau$ is constant. Hence the law of $B$ is a point mass, so $B=\tau(A)$ a.s. and the claim is trivial. Assume henceforth that $\Lambda>0$.

Let $f$ be a convex potential such that $f'=\tau$; since $\tau$ is monotone and $\Lambda$-Lipschitz, we take the continuous representative
\[
 f(t)=f(0)+\int_0^t\tau(s)\,\dd s.
\]
Let $f^*$ be the Legendre transform of $f$.   Kantorovich duality for the quadratic cost gives
\[
 \mathbb E|A-B|^2-W_2^2(\Law(A),\Law(B))
 =
 2\mathbb E\bigl[f(A)+f^*(B)-AB\bigr].
\]
The expression in brackets is called the Fenchel gap. We use the elementary consequence of convex $C^{1,1}$ smoothness
\begin{equation}
\label{eq:fenchel-pointwise-smooth}
 f(a)+f^*(b)-ab
 \geq
 \frac1{2\Lambda}|b-f'(a)|^2
 \qquad a,b\in\R.
\end{equation}
Indeed, the upper Taylor bound
\[
 f(x)\le f(a)+f'(a)(x-a)+\frac\Lambda2|x-a|^2
\]
implies, by evaluating the supremum in the definition of $f^*$ at $x=a+(b-f'(a))/\Lambda$,
\[
 f^*(b)
 \geq
 b\left(a+\frac{b-f'(a)}{\Lambda}\right)
 -f(a)-\frac{f'(a)(b-f'(a))}{\Lambda}
 -\frac{|b-f'(a)|^2}{2\Lambda}.
\]
This is exactly \eqref{eq:fenchel-pointwise-smooth}. Applying it with $a=A$ and $b=B$, and using $f'=\tau$, gives
\[
 f(A)+f^*(B)-AB
 \geq
 \frac1{2\Lambda}|B-\tau(A)|^2.
\]
Taking expectations proves \eqref{eq:fenchel-gap}. 
\end{proof}

\subsection{Stability under SPK}

We now prove Theorem~\ref{thm:spk-main}.

\begin{proof}[Proof of Theorem~\ref{thm:spk-main}]
For each direction $\theta$, applying Lemma  \ref{lem:fenchel-gap} with
\[
 A=\theta\cdot X,
 \qquad
 B=\theta\cdot T(X)
\]
gives 
\begin{equation}
\label{eq:fenchel-direction}
 \int |\theta\cdot T(x)-\tau_\theta(\theta\cdot x)|^2\,\dd\mu(x)
 \leq
\Lambda\left(\int |\theta\cdot(T(x)-x)|^2\,\dd\mu(x)
 -W_2^2(\mu_\theta,\nu_\theta) \right)= \Lambda\,g_\theta(T).
\end{equation}
whenever $\Lip(\tau_\theta)\le\Lambda$. Since $h=\tau_\theta$ is an admissible ridge function,
\[
 \inf_h\int|\theta\cdot T(x)-h(\theta\cdot x)|^2\,\dd\mu(x)
 \leq
 \Lambda g_\theta(T).
\]
Integrating in $\theta$ yields
\begin{equation}
\label{eq:ridge-by-deficit}
 \cR_\mu(T)
 \leq
 \Lambda\cD(\mu,\nu).
\end{equation}

Let $S\in {\rm Tan}_\mu$. By the triangle inequality and the SPK inequality,
\[
\begin{aligned}
 \dist_{L^2(\mu)}(T,\cA_d)
 &\leq
 \|T-S\|_{L^2(\mu)}+
 \dist_{L^2(\mu)}(S,\cA_d)  \\
 &\leq
 \|T-S\|_{L^2(\mu)}+
 \kappa_{\mathrm{SPK}}(\mu)^{-1/2}\sqrt{\cR_\mu(S)}.
\end{aligned}
\]
Using Lemma~\ref{lem:ridge-residual-lipschitz} with $S=T+(S-T)$ gives
\[
 \sqrt{\cR_\mu(S)}
 \leq
 \sqrt{\cR_\mu(T)}+
 \sqrt{\cR_\mu(S-T)}
 \leq
 \sqrt{\cR_\mu(T)}+
 \frac1{\sqrt d}\|S-T\|_{L^2(\mu)}.
\]
Combining this with \eqref{eq:ridge-by-deficit} yields
\[
 \dist_{L^2(\mu)}(T,\cA_d)
 \leq
 \sqrt{\frac{\Lambda}{\kappa_{\mathrm{SPK}}(\mu)}}\,\cD(\mu,\nu)^{1/2}
 +
 \left(1+\frac1{\sqrt{d\,\kappa_{\mathrm{SPK}}(\mu)}}\right)
 \|T-S\|_{L^2(\mu)}.
\]
Taking the infimum over $S\in {\rm Tan}_\mu$ proves \eqref{eq:spk-main-stab}. If $T\in {\rm Tan}_\mu$, the tangent defect vanishes, and squaring gives \eqref{eq:spk-main-stab-tangent}.
\end{proof}

\subsection{On the one-dimensional Lipschitz parameter}
\label{subsec:lambda}

The parameter
\[
 \Lambda=\esssup_{\theta\in\Sd}\Lip(\tau_\theta)
\]
should be regarded as a one-dimensional regularity  for the target marginals. It is not part of the high-dimensional SPK spectral gap.   Lemma~\ref{lem:fenchel-gap} uses it only to turn the Fenchel gap into an $L^2$ distance from the monotone graph.

In concrete situations $\Lambda$ can often be bounded independently of the dimension. For instance, suppose the source is the standard Gaussian and the target has density $e^{-W}$ with
\[
 \nabla^2 W\ge \alpha I_d
\]
for some $\alpha>0$. By the Pr\'ekopa theorem, each one-dimensional marginal $\nu_\theta$ is again $\alpha$-strongly log-concave on $\R$. Caffarelli's contraction theorem~\cite{Caffarelli2000,FathiGozlanProdhomme2020} in dimension one then implies that the monotone map from $N(0,1)$ to $\nu_\theta$ satisfies
\begin{equation}
\label{eq:caffarelli-lambda}
 \Lip(\tau_\theta)\le \alpha^{-1/2}
 \qquad\text{for every }\theta.
\end{equation}
Thus in the Gaussian-source case, uniformly log-concave targets have a universal projected Lipschitz scale. In particular, if $\alpha\ge1$, all projected rearrangements are contractions.

There is also a direct one-dimensional density criterion. If $p_\theta$ and $q_\theta$ are the densities of $\mu_\theta$ and $\nu_\theta$, and the monotone map $\tau_\theta$ is differentiable, then
\[
 \tau_\theta'(s)
 =
 \frac{p_\theta(s)}{q_\theta(\tau_\theta(s))}.
\]
Hence a uniform lower bound on the target projected densities along the monotone image, relative to the source projected densities, gives a bound for $\Lambda$. The uniform Lipschitz assumption in Theorem~\ref{thm:spk-main} may therefore be replaced by any condition that gives the strong convexity of the one-dimensional dual potentials required in Lemma~\ref{lem:fenchel-gap}. We use the Lipschitz form for simplicity.

\section{Stability for Gaussian measures}\label{sec:gaussian}

Let $\gamma_d=N(0,I_d)$ be the standard Gaussian and $X\sim\gamma_d$.
Let $A=A^T$ and $b\in\R^d$, and set
\[
 T(x)=Ax+b.
\]
For affine gradient maps the ridge defect can be computed exactly.

\begin{proposition}[Exact Gaussian linear calculation]
\label{prop:gaussian-linear-exact}
Let $A=A^T$. Then
\begin{equation}
\label{eq:linear-ridge-exact}
 \cR_{\gamma_d}(Ax+b)
 =
 \frac{d\Tr(A^2)-(\Tr A)^2}{d(d+2)}.
\end{equation}
Moreover,
\begin{equation}
\label{eq:linear-dist-exact}
 \inf_{\lambda,b_0}\int |Ax+b-(\lambda x+b_0)|^2\,\dd\gamma_d(x)
 =
 \Tr(A^2)-\frac{(\Tr A)^2}{d}.
\end{equation}
Consequently,
\begin{equation}
\label{eq:linear-exact-ratio}
 \cR_{\gamma_d}(Ax+b)
 =
 \frac1{d+2}
 \inf_{\lambda,b_0}\|Ax+b-(\lambda x+b_0)\|_{L^2(\gamma_d)}^2.
\end{equation}
\end{proposition}

\begin{proof}
Fix $\theta\in\Sd$ and decompose
\[
 X=Z\theta+Y,
 \qquad
 Z=\theta\cdot X\sim N(0,1),
 \qquad
 Y\sim N(0,I_d-\theta\otimes\theta),
\]
with $Z$ and $Y$ independent. Then
\[
 \theta\cdot T(X)=Z\theta^TA\theta+\theta^TA Y+\theta\cdot b.
\]
Thus
\[
 \Var(\theta\cdot T(X)\mid \theta\cdot X)
 =
 \mathbb E[(\theta^TA Y)^2]
 =
 |A\theta|^2-(\theta^TA\theta)^2.
\]
Using
\[
 \int_{\Sd}|A\theta|^2\,\dd\sigma(\theta)=\frac{\Tr(A^2)}d
\]
and the fourth-moment formula
\[
 \int_{\Sd}(\theta^TA\theta)^2\,\dd\sigma(\theta)
 =
 \frac{(\Tr A)^2+2\Tr(A^2)}{d(d+2)},
\]
which is the degree-two instance of the standard spherical moment/Funk--Hecke computation recalled in Appendix~\ref{app:tensor} (see also Lemma \ref{lem:grassmannian_algebra}). We obtain \eqref{eq:linear-ridge-exact}. The optimal translation in \eqref{eq:linear-dist-exact} is $b_0=b$, and the optimal scalar is $\lambda=\Tr(A)/d$. This gives \eqref{eq:linear-dist-exact} and \eqref{eq:linear-exact-ratio}.
\end{proof}

\subsection{Sharp SPK for standard Gaussian}

Let ${\rm Tan}_{\gamma_d}$ be the tangent space at $\gamma_d$. We give the
proof in some detail because this is the explicit case in which the SPK
constant is computed.

We use probabilists' Hermite polynomials, Wick tensors, and the Wiener--It\^o chaos decomposition; these standard conventions may be found, for example, in Janson's account of Gaussian Hilbert spaces \cite{Janson1997} or Nualart's text on Malliavin calculus \cite{Nualart2006}. If $A_n\in\mathrm{Sym}^n(\R^d)$, write
\[
 \psi_n(x)=\langle A_n,\Wick{x^{\otimes n}}\rangle
\]
for the corresponding element of the $n$-th homogeneous Wiener chaos. Then
\begin{equation}
\label{eq:wick-gradient-norm}
 \|\nabla\psi_n\|_{L^2(\gamma_d)}^2
 =
 n\,n!\,\|A_n\|^2.
\end{equation}
For each $\theta\in\Sd$ let $P_\theta$ denote conditional expectation with respect to $\theta\cdot X$.

\begin{lemma}[Projection of one Gaussian chaos]
\label{lem:gaussian-chaos-projection}
Let $\psi_n=\langle A_n,\Wick{X^{\otimes n}}\rangle$ with $n\ge1$. Then
\begin{equation}
\label{eq:conditional-chaos-formula}
 P_\theta(\theta\cdot\nabla\psi_n)
 =
 n\,(A_n:\theta^{\otimes n})\,H_{n-1}(\theta\cdot X),
\end{equation}
where $H_{n-1}$ is the one-dimensional probabilists' Hermite polynomial. Consequently,
\begin{equation}
\label{eq:conditional-chaos-norm}
 \|P_\theta(\theta\cdot\nabla\psi_n)\|_{L^2(\gamma_d)}^2
 =
 n\,n!\,(A_n:\theta^{\otimes n})^2.
\end{equation}
\end{lemma}

\begin{proof}
The directional component of the gradient is
\[
 \theta\cdot\nabla\psi_n
 =
 n\,\bigl\langle A_n(\theta,\cdot,\ldots,\cdot),\Wick{X^{\otimes(n-1)}}\bigr\rangle,
\]
where $A_n(\theta,\cdot,\ldots,\cdot)$ denotes contraction of one tensor index against $\theta$. The conditional expectation of a Wick monomial onto the one-dimensional Gaussian variable $Z=\theta\cdot X$ is its orthogonal projection onto the chaos generated by $Z$:
\[
 \mathbb E\bigl[\Wick{X^{\otimes(n-1)}}\mid Z\bigr]
 =
 H_{n-1}(Z)\,\theta^{\otimes(n-1)}.
\]
Contracting with $A_n(\theta,\cdot,\ldots,\cdot)$ gives \eqref{eq:conditional-chaos-formula}. Since
\[
 \|H_{n-1}(Z)\|_{L^2}^2=(n-1)!,
\]
we obtain
\[
 \|P_\theta(\theta\cdot\nabla\psi_n)\|_2^2
 =
 n^2(A_n:\theta^{\otimes n})^2 (n-1)!
 =
 n\,n!\,(A_n:\theta^{\otimes n})^2.
\]
\end{proof}

\begin{lemma}[Chaos contribution to the ridge defect]
\label{lem:gaussian-chaos-ridge}
For $\psi_n=\langle A_n,\Wick{X^{\otimes n}}\rangle$,
\begin{equation}
\label{eq:chaos-ridge-expanded}
 \int_{\Sd}\|(I-P_\theta)(\theta\cdot\nabla\psi_n)\|_{L^2(\gamma_d)}^2\,\dd\sigma(\theta)
 =
 n\,n!\left[
 \frac1d\|A_n\|^2
 -
 \int_{\Sd}(A_n:\theta^{\otimes n})^2\,\dd\sigma(\theta)
 \right].
\end{equation}
\end{lemma}

\begin{proof}
First,
\[
 \int_{\Sd}\|\theta\cdot\nabla\psi_n\|_{L^2(\gamma_d)}^2\,\dd\sigma(\theta)
 =
 \frac1d\|\nabla\psi_n\|_{L^2(\gamma_d)}^2
 =
 \frac{n\,n!}{d}\|A_n\|^2
\]
by \eqref{eq:wick-gradient-norm} and the normalization of $\sigma$. Second, conditional expectation is an orthogonal projection, so
\[
 \|(I-P_\theta)(\theta\cdot\nabla\psi_n)\|_2^2
 =
 \|\theta\cdot\nabla\psi_n\|_2^2
 -
 \|P_\theta(\theta\cdot\nabla\psi_n)\|_2^2.
\]
Integrating in $\theta$ and using Lemma~\ref{lem:gaussian-chaos-projection} gives \eqref{eq:chaos-ridge-expanded}.
\end{proof}

\begin{lemma}[Polynomial-gradient density for the Gaussian tangent space]
\label{lem:gaussian-polynomial-density}
The space of polynomial gradients is dense in $\operatorname{Tan}_{\gamma_d}$ with respect to the $L^2(\gamma_d;\R^d)$ norm.
\end{lemma}

\begin{proof}
Lemma~\ref{lem:tan-density-criterion} shows first that every polynomial gradient belongs to $\operatorname{Tan}_{\gamma_d}$. Conversely, it is enough by definition of $\operatorname{Tan}_{\gamma_d}$ to approximate $\nabla f$ for $f\in C_c^\infty(\R^d)$. Expand $f$ in the Hermite basis,
\[
 f=\sum_{n\ge0} f_n,
\]
where $f_n$ is the $n$-th homogeneous Wiener-chaos component. Since $f\in W^{1,2}(\gamma_d)$, the Ornstein--Uhlenbeck spectral identity gives
\[
 \|\nabla f\|_{L^2(\gamma_d)}^2
 =
 \sum_{n\ge1} n\|f_n\|_{L^2(\gamma_d)}^2<\infty.
\]
Therefore the partial sums $S_N f=\sum_{n=0}^N f_n$, which are polynomials, satisfy
\[
 \|\nabla S_N f-\nabla f\|_{L^2(\gamma_d)}^2
 =
 \sum_{n>N} n\|f_n\|_{L^2(\gamma_d)}^2
 \longrightarrow0.
\]
Thus polynomial gradients are dense in the closure of $C_c^\infty$-gradients, namely in $\operatorname{Tan}_{\gamma_d}$.
\end{proof}

Before the proof of Theorem~\ref{thm:gaussian-sk-intro}, we explain the role of the low Hermite chaoses.
The decomposition below is with respect to the homogeneous chaos of the
potential \(\psi\), so that \(u=\nabla\psi\).

\begin{center}
\renewcommand{\arraystretch}{1.2}
\begin{tabular}{lll}
\hline
\textbf{Component} & \textbf{Vector field} & \textbf{Role} \\
\hline
\(n=1\)
&
\(u=b\)
&
translations in \(\mathcal A_d\)
\\[1mm]
\(n=2\), scalar trace
&
\(u=\lambda x\)
&
dilation in \(\mathcal A_d\)
\\[1mm]
\(n=2\), trace-free
&
\(u=Ax,\ A=A^{\mathsf T},\ \operatorname{Tr}A=0\)
&
positive defect, not sharp
\\[1mm]
\(n=3\), trace component
&
\(I\odot\mathcal H_1\)
&
sharp non-rigid mode
\\[1mm]
higher components
&
remaining Hermite components
&
controlled by Appendix~\ref{app:tensor}
\\
\hline
\end{tabular}
\end{center}
\begin{proof}[Proof of Theorem~\ref{thm:gaussian-sk-intro}]
We first assume $u=\nabla\psi$ with $\psi$ a polynomial. Write its homogeneous Wiener-chaos expansion
\[
 \psi=\sum_{n\ge0}\psi_n,
 \qquad
 \psi_n\in \text{the }n\text{-th chaos}.
\]
Then
\[
 u=\sum_{n\ge1}\nabla\psi_n,
\]
and these vector fields are mutually orthogonal in $L^2(\gamma_d;\R^d)$. For fixed $\theta$, the conditional expectation $P_\theta$ preserves homogeneous chaoses, because it is the orthogonal projection onto the closed subspace generated by the one-dimensional Gaussian variable $\theta\cdot X$. Hence the quantities
\[
 (I-P_\theta)(\theta\cdot\nabla\psi_n)
\]
are orthogonal in $L^2(\gamma_d)$ for different $n$. Therefore
\begin{equation}
\label{eq:gauss-R-sum}
 \cR_{\gamma_d}(u)
 =
 \sum_{n\ge1}
 \int_{\Sd}\|(I-P_\theta)(\theta\cdot\nabla\psi_n)\|_2^2\,\dd\sigma(\theta).
\end{equation}

Let $\psi_n=\langle A_n,\Wick{X^{\otimes n}}\rangle$. Lemma~\ref{lem:gaussian-chaos-ridge} gives the contribution of the $n$-th chaos. The rigid modes are exactly the following: $n=1$ gives constant vector fields $b$, and the scalar trace part of $n=2$ gives fields $\lambda x$. These modes span $\cA_d$ and contribute zero to the ridge defect. All other chaos components are orthogonal to $\cA_d$.

For the non-rigid components, Lemma~\ref{lem:tensor-estimate} yields
\[
 \int_{\Sd}(A_n:\theta^{\otimes n})^2\,\dd\sigma(\theta)
 \leq
 \frac{3}{d(d+2)}\|A_n\|^2,
\]
with the sharper value $2/[d(d+2)]$ on the trace-free part of $n=2$. Inserting this into \eqref{eq:chaos-ridge-expanded} gives, for every non-rigid component,
\[
 \int_{\Sd}\|(I-P_\theta)(\theta\cdot\nabla\psi_n)\|_2^2\,\dd\sigma(\theta)
 \geq
 \left(\frac1d-\frac{3}{d(d+2)}\right)n\,n!\|A_n\|^2.
\]
By \eqref{eq:wick-gradient-norm}, this is
\[
 \geq
 \frac{d-1}{d(d+2)}\|\nabla\psi_n\|_{L^2(\gamma_d)}^2.
\]
Summing over the orthogonal complement of the rigid modes and using \eqref{eq:gauss-R-sum}, we obtain
\[
 \cR_{\gamma_d}(u)
 \geq
 \frac{d-1}{d(d+2)}
 \dist_{L^2(\gamma_d)}^2(u,\cA_d).
\]

For a general $u\in {\rm Tan}_{\gamma_d}$, choose polynomial gradients $u_j$ converging to $u$ in $L^2(\gamma_d;\R^d)$ by Lemma~\ref{lem:gaussian-polynomial-density}. The distance to the finite-dimensional closed space $\cA_d$ is continuous in $L^2$, and the ridge defect is continuous by Lemma~\ref{lem:ridge-continuity}. Passing to the limit gives \eqref{eq:gaussian-sk-intro}.

It remains to prove sharpness. Appendix~\ref{app:tensor} shows that in degree $3$ the trace component $I\odot \mathcal H_1\subset\mathrm{Sym}^3(\R^d)$ attains
\[
 \int_{\Sd}(A_3:\theta^{\otimes3})^2\,\dd\sigma(\theta)
 =
 \frac{3}{d(d+2)}\|A_3\|^2
\]
for every non-zero $A_3\in I\odot\mathcal H_1$. Let
\[
 \psi_3(x)=\langle A_3,\Wick{x^{\otimes3}}\rangle,
 \qquad
 u=\nabla\psi_3.
\]
Then $u$ is a polynomial gradient, hence $u\in\operatorname{Tan}_{\gamma_d}$, and it is orthogonal to $\cA_d$ because it lies in the vector-valued second Wiener chaos. Lemma~\ref{lem:gaussian-chaos-ridge} and \eqref{eq:wick-gradient-norm} give
\[
 \frac{\cR_{\gamma_d}(u)}{\dist_{L^2(\gamma_d)}^2(u,\cA_d)}
 =
 \frac{\cR_{\gamma_d}(u)}{\|u\|_{L^2(\gamma_d)}^2}
 =
 \frac1d-\frac{3}{d(d+2)}
 =
 \frac{d-1}{d(d+2)}.
\]
Thus no larger SPK constant can hold.
\end{proof}

\begin{lemma}[Spherical tensor estimate]
\label{lem:tensor-estimate}
Let $A$ be a symmetric $n$-tensor on $\R^d$. If either $n\ge3$, or $n=2$ and $A$ is trace-free, then
\[
 \int_{\Sd}(A:\theta^{\otimes n})^2\,\dd\sigma(\theta)
 \leq
 \frac{3}{d(d+2)}\|A\|^2.
\]
For $n\ge3$ this constant is optimal; equality already occurs in degree $3$ on the trace component $I\odot\mathcal H_1$. For $n=2$ and $\Tr A=0$ one has the sharper identity
\[
 \int_{\Sd}(\theta^TA\theta)^2\,\dd\sigma(\theta)
 =
 \frac{2}{d(d+2)}\|A\|_{\mathrm{HS}}^2.
\]
\end{lemma}

\begin{proof}
The trace-free $n=2$ identity is the usual fourth-moment formula on the sphere. The general estimate is a finite-dimensional Funk--Hecke, or equivalently representation-theoretic, calculation for symmetric tensors; see Appendix~\ref{app:tensor} and the references therein for the normalization. The appendix shows that the eigenvalues of the operator
\[
 A\mapsto \int_{\Sd}(A:\theta^{\otimes n})\theta^{\otimes n}\,\dd\sigma(\theta)
\]
on the irreducible trace components of $\mathrm{Sym}^n(\R^d)$ are bounded by $3/[d(d+2)]$ except for the scalar trace component in degree $n=2$, which corresponds exactly to the rigid field $x\mapsto \lambda x$. The same appendix also shows that the value $3/[d(d+2)]$ is attained on $I\odot\mathcal H_1$ in degree $3$. This proves the lemma.
\end{proof}

\begin{remark}[Sharpness and affine modes]
Proposition~\ref{prop:gaussian-linear-exact} shows that trace-free linear gradient fields have ratio $1/(d+2)$. This is larger than the sharp global value $(d-1)/[d(d+2)]$ when $d>1$. Thus the extremizers for the full Gaussian SPK constant are not affine trace-free maps. They occur already in the third Hermite chaos, on the trace component $I\odot\mathcal H_1$ described in Appendix~\ref{app:tensor}. The affine calculation is included because it identifies the exact behavior on linear Brenier maps, but it is not the worst Gaussian spectral mode.
\end{remark}

\medskip

\begin{corollary}[Gaussian stability]
\label{cor:gaussian-stability}
Let $T=\nabla\varphi\in {\rm Tan}_{\gamma_d}$  be the optimal transport map from $\gamma_d$ to $\nu$. Let $\tau_\theta$ be the monotone transport map from $(\gamma_d)_\theta$ to $\nu_\theta$, and suppose
\[
 \Lambda:=\esssup_\theta\Lip(\tau_\theta)<\infty.
\]
Then
\[
 \dist_{L^2(\gamma_d)}^2(T,\cA_d)
 \leq
 \frac{d(d+2)}{d-1}\,\Lambda\,\cD(\gamma_d,\nu).
\]
\end{corollary}

\begin{proof}
For each $\theta$, \eqref{eq:fenchel-direction} gives
\[
 \int |\theta\cdot T(x)-\tau_\theta(\theta\cdot x)|^2\,\dd\gamma_d(x)
 \leq
 \Lambda g_\theta(T).
\]
Since $h=\tau_\theta$ is admissible in the definition of $\cR_{\gamma_d}(T)$,
\[
 \cR_{\gamma_d}(T)\leq \Lambda\int g_\theta(T)\,\dd\sigma(\theta)
 =\Lambda\cD(\gamma_d,\nu).
\]
Theorem~\ref{thm:gaussian-sk-intro} completes the proof.
\end{proof}
\subsection{Isotropic Gaussian measures}

\begin{proposition}[Isotropic Gaussian measures]
\label{prop:spherical-gaussian-spk}
Let
\[
 \gamma_{a,\sigma}:=N(a,\sigma^2 I_d),
 \qquad a\in\R^d,\quad \sigma>0.
\]
Then
\[
 \kappa_{\mathrm{SPK}}(\gamma_{a,\sigma})
 =
 \frac{d-1}{d(d+2)}.
\]
Equivalently, the normalized sliced Poincar\'e--Korn constant satisfies
\[
 \overline{\kappa}_{\mathrm{SPK}}(\gamma_{a,\sigma})
 =
 d\,\kappa_{\mathrm{SPK}}(\gamma_{a,\sigma})
 =
 \frac{d-1}{d+2}.
\]
Consequently, the same quantitative stability estimate holds as in the standard Gaussian case.
\end{proposition}

\begin{proof}
The case \(a=0\) and \(\sigma=1\) is exactly the Gaussian sliced Poincar\'e--Korn
inequality proved in Theorem~\ref{thm:gaussian-sk-intro}. We reduce the general
spherical case to the standard one.

Let \(U:L^2(\gamma_{a,\sigma};\R^d)\to L^2(\gamma_d;\R^d)\) be the isometry
\[
 (Uu)(z):=u(a+\sigma z).
\]
This map identifies the two tangent spaces. Indeed, if \(u=\nabla_x\psi\) with \(\psi\in C_c^\infty(\R^d)\), then
\[
 (Uu)(z)
 =
 \nabla_z\left(\frac1\sigma\psi(a+\sigma z)\right),
\]
and the converse follows by applying the inverse affine change of variables. Thus the identities below, first checked on smooth compactly supported gradients, extend by closure to all tangent fields because both \(\cR\) and the distance to the finite-dimensional space \(\cA_d\) are continuous in \(L^2\).

Let \(X\sim\gamma_{a,\sigma}\). Write
\[
 X=a+\sigma Z,
 \qquad Z\sim\gamma_d.
\]
For \(u\in {\rm Tan}_{\gamma_{a,\sigma}}\), define
\[
 v(z):=u(a+\sigma z).
\]
Then \(v\in {\rm Tan}_{\gamma_d}\).

For every \(\theta\in\Sd\),
\[
 \theta\cdot X
 =
 \theta\cdot a+\sigma\,\theta\cdot Z.
\]
Thus conditioning on \(\theta\cdot X\) is equivalent to conditioning on
\(\theta\cdot Z\). Hence
\[
 \cR_{\gamma_{a,\sigma}}(u)
 =
 \cR_{\gamma_d}(v).
\]

Moreover,
\[
 \dist_{L^2(\gamma_{a,\sigma})}^2(u,\cA_d)
 =
 \dist_{L^2(\gamma_d)}^2(v,\cA_d).
\]
Indeed,
\[
 \lambda(a+\sigma z)+b
 =
 (\lambda\sigma)z+(\lambda a+b),
\]
and as \(\lambda\in\R\) and \(b\in\R^d\) vary, the pair
\[
 \alpha=\lambda\sigma,\qquad \beta=\lambda a+b
\]
runs over all \(\alpha\in\R\) and \(\beta\in\R^d\).

Applying Theorem~\ref{thm:gaussian-sk-intro} to \(v\) under \(\gamma_d\), we obtain
\[
 \cR_{\gamma_d}(v)
 \geq
 \frac{d-1}{d(d+2)}
 \dist_{L^2(\gamma_d)}^2(v,\cA_d).
\]
Using the two identities above gives
\[
 \cR_{\gamma_{a,\sigma}}(u)
 \geq
 \frac{d-1}{d(d+2)}
 \dist_{L^2(\gamma_{a,\sigma})}^2(u,\cA_d).
\]
Therefore translations and scalar dilations do not
change the sliced Poincar\'e--Korn constant. Since Theorem~\ref{thm:gaussian-sk-intro} gives the sharp value for \(\gamma_d\),  we have
\[
 \kappa_{\mathrm{SPK}}(\gamma_{a,\sigma})
 =
 \frac{d-1}{d(d+2)}.
\]

For the Wasserstein stability estimate, if
\[
S(z)=\frac{T(a+\sigma z)-a}{\sigma},
\]
then
\[
\dist_{L^2(\gamma_{a,\sigma})}^2(T,\cA_d)
=
\sigma^2
\dist_{L^2(\gamma_d)}^2(S,\cA_d),
\]
and
\[
\cD\bigl(\gamma_{a,\sigma},T_\#\gamma_{a,\sigma}\bigr)
=
\sigma^2\cD\bigl(\gamma_d,S_\#\gamma_d\bigr).
\]
The projected Lipschitz scale \(\Lambda\) is also unchanged under this affine
normalization.  Thus the stability constant is unchanged.
\end{proof}

\section{Examples and obstructions}\label{example}
\subsection{Perturbations of the Gaussian}

\begin{proposition}[Bounded Gaussian perturbations]
\label{prop:bounded-perturbation}
Assume
\[
 \dd\mu=\rho\dd\gamma_d,
 \qquad
 0<m\leq \rho\leq M<\infty.
\]
Then
\[
 \kappa_{\mathrm{SPK}}(\mu)
 \geq
 \frac{m}{M}\frac{d-1}{d(d+2)}.
\]
Consequently, for Brenier maps $T:\mu\to\nu$ with $T\in{\rm Tan}_\mu$ and satisfying $\esssup_\theta\Lip(\tau_\theta)\le\Lambda$,
\[
 \dist_{L^2(\mu)}^2(T,\cA_d)
 \leq
 \frac{M}{m}\frac{d(d+2)}{d-1}\,\Lambda\,\cD(\mu,\nu).
\]
\end{proposition}

\begin{proof}
Since $m\le\rho\le M$, the $L^2(\mu)$ and $L^2(\gamma_d)$ norms are equivalent. Hence the closures of $\{\nabla\phi:\phi\in C_c^\infty(\R^d)\}$ in these two norms coincide, so ${\rm Tan}_\mu={\rm Tan}_{\gamma_d}$ as sets. For each \(\theta\), if \(Z=\theta\cdot X\) with \(X\sim\gamma_d\), then
 \(m\le \dd\mu_\theta/\dd\gamma_1\le M\). Thus
\(L^2(\mu_\theta)\) and \(L^2(\gamma_1)\) coincide as sets with equivalent norms.

For every $u\in{\rm Tan}_\mu={\rm Tan}_{\gamma_d}$, every direction, and every ridge function $h(\theta\cdot x)$,
\[
 \int|\theta\cdot u-h(\theta\cdot x)|^2\,\dd\mu
 \geq
 m\int|\theta\cdot u-h(\theta\cdot x)|^2\,\dd\gamma_d.
\]
Taking the infimum over $h$ and integrating in $\theta$ gives
\[
 \cR_\mu(u)\geq m\cR_{\gamma_d}(u).
\]
On the other hand,
\[
 \dist_{L^2(\mu)}^2(u,\cA_d)
 \leq
 M\dist_{L^2(\gamma_d)}^2(u,\cA_d).
\]
The Gaussian sliced Poincar\'e--Korn inequality gives the claim.
\end{proof}

\subsection{Fixed measures and compact classes}

The next criterion gives a fixed-measure SPK bound on compact classes. The
constant is not dimension-free and is generally non-explicit.

\begin{proposition}[Compactness criterion]
\label{prop:compactness-spk}
Let $\mu=\rho\cL^d$ on a connected open set $\Omega$ with $\mu(\Omega)=1$, and assume that $\rho$ is locally bounded below on $\Omega$. Let $\mathcal C\subset {\rm Tan}_\mu$ be a class of gradient fields. For each $u\in\mathcal C$, let $a_u\in\cA_d$ be a nearest point to $u$ in $\cA_d$. Assume that the normalized residual set
\[
 \mathcal K
 :=
 \left\{
 \frac{u-a_u}{\dist_{L^2(\mu)}(u,\cA_d)}:
 u\in\mathcal C,\quad u\notin\cA_d
 \right\}
\]
is relatively compact in $L^2(\mu;\R^d)$. Then there exists $\kappa(\mu,\mathcal C)>0$ such that
\[
 \dist_{L^2(\mu)}^2(u,\cA_d)
 \leq
 \frac1{\kappa(\mu,\mathcal C)}\cR_\mu(u)
 \qquad
 \forall u\in\mathcal C.
\]
\end{proposition}

\begin{proof}
For every $a\in\cA_d$, the scalar function $\theta\cdot a(x)$ has the form $\lambda\theta\cdot x+\theta\cdot b$ and is therefore a ridge function of $\theta\cdot x$. Hence
\[
 \cR_\mu(u+a)=\cR_\mu(u),
 \qquad
 \cR_\mu(cu)=c^2\cR_\mu(u).
\]
Suppose the conclusion fails. Then there are $u_n\in\mathcal C\setminus\cA_d$ such that
\[
 \frac{\cR_\mu(u_n)}{\dist_{L^2(\mu)}^2(u_n,\cA_d)}\longrightarrow0.
\]
Set $\delta_n:=\dist_{L^2(\mu)}(u_n,\cA_d)$ and
\[
 v_n:=\frac{u_n-a_{u_n}}{\delta_n}.
\]
Then $v_n\in\mathcal K$ and
\[
 \dist_{L^2(\mu)}(v_n,\cA_d)=1,
 \qquad
 \cR_\mu(v_n)=\frac{\cR_\mu(u_n)}{\delta_n^2}\longrightarrow0.
\]
By the relative compactness of $\mathcal K$, a subsequence converges in $L^2(\mu)$ to some $v$.

By Lemma~\ref{lem:ridge-continuity}, the ridge defect is continuous on $L^2(\mu;\R^d)$, and the distance to the finite-dimensional closed space $\cA_d$ is continuous. Hence
\[
 \cR_\mu(v)=0,
 \qquad
 \dist_{L^2(\mu)}(v,\cA_d)=1.
\]
Since \(\rho\) is locally bounded below on \(\Omega\), convergence in
\(L^2(\mu)\) implies convergence in \(L^2_{\rm loc}(\Omega)\). Hence the
\(L^2(\mu)\)-closure of smooth gradients consists, locally in \(\Omega\),
of vector fields with symmetric distributional derivative.
By approximation with smooth gradients, each \(v_n\), and hence the limit
\(v\), has symmetric distributional derivative. Then Lemma~\ref{lem:ridge-kernel-gradient} gives $v\in\cA_d$, a contradiction.
\end{proof}

\begin{remark}[Compact classes]
The proposition is a fixed-measure compactness device, not a uniform geometric
estimate. It applies to classes whose normalized residual directions are
precompact in \(L^2(\mu)\), for instance finite-dimensional ansatz classes or
classes with uniform compactness estimates after \(L^2\)-normalization. The
resulting constant \(\kappa(\mu,\mathcal C)\) is generally non-explicit and
depends on the chosen class.
\end{remark}

\subsection{An anisotropic Gaussian counterexample}
\label{sec:counterexamples}

We now prove Proposition~\ref{prop:aniso-obstruction-intro}. Let
\[
 \mu_\varepsilon=N\left(0,
 \begin{pmatrix}
 \varepsilon^2&0\\
 0&1
 \end{pmatrix}
 \right),
 \qquad
 0<\varepsilon<1.
\]
Then $\mu_\varepsilon=e^{-V_\varepsilon}\dd x/Z_\varepsilon$ with
\[
 V_\varepsilon(x_1,x_2)=\frac{x_1^2}{2\varepsilon^2}+\frac{x_2^2}{2},
 \qquad
 \nabla^2V_\varepsilon=\begin{pmatrix}
 \varepsilon^{-2}&0\\0&1
 \end{pmatrix}
 \geq I.
\]
Thus the Bakry--\'Emery curvature lower bound is uniform. The Poincar\'e constant is also uniformly bounded by $1$.

Consider the gradient field
$
 u(x_1,x_2)=(x_2,x_1)=\nabla(x_1x_2).
$ By the same Gaussian cut-off argument as in Lemma~\ref{lem:tan-density-criterion},
the polynomial gradient \(u=\nabla(x_1x_2)\) belongs to
\({\rm Tan}_{\mu_\varepsilon}\).

Since $\mathbb E u=0$,  and for any $X=(X_1, X_2)\sim \mu_\varepsilon$ it holds
\[
 \mathbb E[X\cdot u(X)]=2\mathbb E[X_1X_2]=0,
\]
this field is orthogonal to $\cA_2$ in $L^2(\mu_\varepsilon)$. Therefore
\begin{equation}
\label{eq:aniso-left}
 \dist_{L^2(\mu_\varepsilon)}^2(u,\cA_2)
 =
 \int |u|^2\,\dd\mu_\varepsilon
 =
 1+\varepsilon^2.
\end{equation}
Let $\theta=(\cos t,\sin t)$. Then
\[
 \theta\cdot X=\cos t\,X_1+\sin t\,X_2,
 \qquad
 \theta\cdot u(X)=\cos t\,X_2+\sin t\,X_1.
\]
For jointly Gaussian variables,
\[
 \Var(Y\mid Z)=\Var(Y)-\frac{\Cov(Y,Z)^2}{\Var(Z)}.
\]
A direct calculation gives
\begin{equation}
\label{eq:aniso-condvar}
 \Var\bigl(\theta\cdot u(X)\mid\theta\cdot X\bigr)
 =
 \frac{\varepsilon^2\cos^2(2t)}{\varepsilon^2\cos^2t+\sin^2t}.
\end{equation}
Hence
\begin{align}
 \cR_{\mu_\varepsilon}(u)
 &=
 \frac1{2\pi}\int_0^{2\pi}
 \frac{\varepsilon^2\cos^2(2t)}{\varepsilon^2\cos^2t+\sin^2t}\,\dd t \\
 &\leq
 \frac1{2\pi}\int_0^{2\pi}
 \frac{\varepsilon^2}{\varepsilon^2\cos^2t+\sin^2t}\,\dd t
 =
 \varepsilon.
\end{align}
Together with \eqref{eq:aniso-left}, this gives
\[
 \frac{\cR_{\mu_\varepsilon}(u)}{\dist_{L^2(\mu_\varepsilon)}^2(u,\cA_2)}
 \leq
 \frac{\varepsilon}{1+\varepsilon^2}\to0.
\]
Thus $\kappa_{\mathrm{SPK}}(\mu_\varepsilon)\to0$.

\bigskip

The preceding SPK degeneration gives a stability counterexample. The
computation is explicit because all projected marginals are one-dimensional
Gaussians.
Let
\[
 S=\begin{pmatrix}0&1\\1&0\end{pmatrix},
 \qquad
 T_\delta(x)=(I+\delta S)x.
\]
For $|\delta|<1$, the matrix $I+\delta S$ is positive definite, so
\[
 T_\delta=\nabla\left(\frac12|x|^2+\delta x_1x_2\right)
\]
is a Brenier map. Let $\nu_{\varepsilon,\delta}=(T_\delta)_\#\mu_\varepsilon$. Then
\[
 T_\delta(x)-x=\delta u(x),
 \qquad
 \dist_{L^2(\mu_\varepsilon)}^2(T_\delta,\cA_2)
 =\delta^2(1+\varepsilon^2).
\]
Fix $\theta=(\cos t,\sin t)$ and put
\[
 Z=\theta\cdot X,
 \qquad
 Y=\theta\cdot u(X),
 \qquad X\sim \mu_\varepsilon.
\]
Set
\[
 a=\Var Z=\varepsilon^2\cos^2t+\sin^2t,
 \quad
 c=\Cov(Z,Y)=(1+\varepsilon^2)\sin t\cos t,
 \quad
 q=\Var(Y\mid Z).
\]
Then the variance of $\theta\cdot T_\delta(X)=Z+\delta Y$ is
\[
 b_\delta=a+2\delta c+\delta^2\Var Y.
\]
Since the Wasserstein distance between centered one-dimensional Gaussians of variances $a$ and $b_\delta$ is $(\sqrt{b_\delta}-\sqrt a)^2$, the directional deficit is
\begin{equation}
\label{eq:explicit-gaussian-directional-deficit}
 g_{\theta}(\delta)
 =\delta^2 \Var Y
 -(\sqrt{b_\delta}-\sqrt a)^2.
\end{equation}
Because $a\ge \varepsilon^2>0$, Taylor expansion in $\delta$ is uniform in $t$ for each fixed $\varepsilon$. From \eqref{eq:explicit-gaussian-directional-deficit},
\[
 \lim_{\delta\to0}
 \frac{g_{\theta}(\delta)}{\delta^2}
 =
 \Var Y-\frac{c^2}{a}
 =\Var(Y\mid Z).
\]
The last quantity is exactly the integrand in \eqref{eq:aniso-condvar}. Dominated convergence therefore gives the exact second variation
\begin{equation}
\label{eq:aniso-second-variation}
 \lim_{\delta\to0}
 \frac{\cD(\mu_\varepsilon,\nu_{\varepsilon,\delta})}{\delta^2}
 =
 \cR_{\mu_\varepsilon}(u).
\end{equation}
Consequently, for every fixed $\varepsilon$ and all sufficiently small $\delta$,
\[
 \cD(\mu_\varepsilon,\nu_{\varepsilon,\delta})
 \leq
 2\delta^2 \cR_{\mu_\varepsilon}(u)
 \leq
 2\delta^2\varepsilon.
\]
It follows that
\[
 \frac{\dist_{L^2(\mu_\varepsilon)}^2(T_\delta,\cA_2)}
 {\cD(\mu_\varepsilon,\nu_{\varepsilon,\delta})}
 \geq
 \frac{1+\varepsilon^2}{2\varepsilon}
 \longrightarrow \infty.
\]
Moreover, if $\delta=o(\varepsilon)$, the projected one-dimensional Gaussian maps have Lipschitz constants uniformly bounded. Indeed, their Lipschitz factors are $\sqrt{b_\delta/a}$, and the lower bound $a\ge \varepsilon^2$ gives $b_\delta/a=1+O(\delta/\varepsilon)+O(\delta^2/\varepsilon^2)$ uniformly in $t$. Thus the obstruction persists even when the one-dimensional scale parameter $\Lambda$ is kept bounded.

Therefore no stability theorem of the form
\[
 \dist^2(T,\cA_d)
 \leq C\cD(\mu,\nu)
\]
can hold with $C$ depending only on a Bakry--\'Emery lower curvature bound or on a usual Poincar\'e constant.

\appendix

\section{The spherical tensor estimate}
\label{app:tensor}

We prove the tensor estimate used in Lemma~\ref{lem:tensor-estimate}. The
argument decomposes symmetric tensors into trace components and applies
Schur's lemma to \(O(d)\)-invariant quadratic forms. Related diagonalizations
appear in \cite{Higuchi1987,JamesConstantine1974}. Let
$\mathrm{Sym}^n(\R^d)$ be the space of symmetric $n$-tensors with the
Hilbert--Schmidt inner product. For $A\in\mathrm{Sym}^n(\R^d)$ define
\[
 Q_n(A):=
 \int_{\Sd}(A:\theta^{\otimes n})^2\,\dd\sigma(\theta).
\]
The quadratic form $Q_n$ is invariant under the orthogonal group $O(d)$. The standard irreducible decomposition of symmetric tensors is
\begin{equation}
\label{eq:trace-decomp-app}
 \mathrm{Sym}^n(\R^d)
 =
 \bigoplus_{j=0}^{\lfloor n/2\rfloor}
 I^{\odot j}\odot \mathcal H_{n-2j},
\end{equation}
where $\mathcal H_m$ denotes the trace-free symmetric $m$-tensors, equivalently harmonic homogeneous polynomials of degree $m$, and $\odot$ is normalized so that
\[
 (I^{\odot j}\odot H):\theta^{\otimes n}
 =
 H:\theta^{\otimes m}
 \qquad (m=n-2j,\ |\theta|=1).
\]
By Schur's lemma, $Q_n$ is diagonal on \eqref{eq:trace-decomp-app}. If
\[
 A=I^{\odot j}\odot H,
 \qquad H\in\mathcal H_m,
 \qquad n=m+2j,
\]
then a direct contraction calculation gives the eigenvalue
\begin{equation}
\label{eq:tensor-eigenvalue}
 \frac{Q_n(A)}{\|A\|^2}
 =
 \lambda_{m,j}^{(d)}
 :=
 \frac{(m+2j)!\,\Gamma(d/2)}
 {2^{m+2j}j!\,\Gamma(m+j+d/2)}.
\end{equation}

\medskip
We use two ingredients in this calculation. First, for $H\in\mathcal H_m$,
\begin{equation}
\label{eq:harmonic-sphere-norm}
 \int_{\Sd}(H:\theta^{\otimes m})^2\,\dd\sigma(\theta)
 =
 \frac{m!\Gamma(d/2)}{2^m\Gamma(m+d/2)}\|H\|^2.
\end{equation}
Second, with the above normalization of $I^{\odot j}\odot H$,
\begin{equation}
\label{eq:trace-norm-factor}
 \|I^{\odot j}\odot H\|^2
 =
 \frac{m!\,2^{2j}j!\,\Gamma(m+j+d/2)}
 {(m+2j)!\,\Gamma(m+d/2)}\|H\|^2.
\end{equation}
Dividing \eqref{eq:harmonic-sphere-norm} by \eqref{eq:trace-norm-factor} gives \eqref{eq:tensor-eigenvalue}.
For completeness we recall the normalization calculation. Let
\(h(x)=H:x^{\otimes m}\). Since \(H\) is trace-free, \(h\) is harmonic and
\[
 \mathbb E h(G)^2=m!\|H\|^2,
 \qquad G\sim N(0,I_d).
\]
Writing \(G=R\Theta\), with \(\Theta\sim\sigma\) and
\(\mathbb E R^{2m}=2^m\Gamma(m+d/2)/\Gamma(d/2)\), gives
\[
 \int_{\mathbb S^{d-1}} h(\theta)^2\,d\sigma(\theta)
 =
 \frac{m!\Gamma(d/2)}{2^m\Gamma(m+d/2)}\|H\|^2.
\]

Next set
\[
 A=I^{\odot j}\odot H,
 \qquad
 p_A(x)=A:x^{\otimes(m+2j)}=|x|^{2j}h(x).
\]
Using the Fischer inner product,
\[
 \|A\|^2=\frac1{(m+2j)!}\,p_A(\partial)p_A(x)\big|_{x=0}.
\]
Since \(h\) is harmonic,
\[
 \Delta^j\bigl(|x|^{2j}h(x)\bigr)
 =
 4^j j!\,(m+d/2)_j\,h(x).
\]
Therefore
\[
 \|A\|^2
 =
 \frac{4^j j!(m+d/2)_j}{(m+2j)!}
 h(\partial)h(x)\big|_{x=0}
 =
 \frac{m!\,2^{2j}j!\,\Gamma(m+j+d/2)}
 {(m+2j)!\,\Gamma(m+d/2)}\|H\|^2.
\]

\medskip

The rigid scalar mode in degree two is $(m,j)=(0,1)$; its eigenvalue is $1/d$ and corresponds to the homothetic affine field. It is excluded from the non-rigid part of the Gaussian SPK inequality. The trace-free quadratic mode is $(m,j)=(2,0)$, and
\[
 \lambda_{2,0}^{(d)}
 =
 \frac{2}{d(d+2)}.
\]
For later reference, the fixed-$m$ ratio is
\[
 \frac{\lambda_{m,j+1}^{(d)}}{\lambda_{m,j}^{(d)}}
 =
 \frac{(m+2j+2)(m+2j+1)}{4(j+1)(m+j+d/2)}.
\]
To locate the maximum in each fixed degree $n$, however, it is convenient to
write $m=n-2j$ and
\[
 \lambda_{n-2j,j}^{(d)}
 =
 \frac{n!\,\Gamma(d/2)}{2^n j!\,\Gamma(n-j+d/2)}.
\]
For fixed $n$, consecutive admissible trace components satisfy
\[
 \frac{\lambda_{n-2(j+1),j+1}^{(d)}}{\lambda_{n-2j,j}^{(d)}}
 =
 \frac{n-j-1+d/2}{j+1}.
\]
If $0\le j<\lfloor n/2\rfloor$ and $d\ge2$, this ratio is at least one. Thus the largest component in degree $n$ is the most traced one: $(m,j)=(0,n/2)$ for even $n$ and $(m,j)=(1,(n-1)/2)$ for odd $n$.

Along the even branch $n=2q$,
\[
 \frac{\lambda_{0,q+1}^{(d)}}{\lambda_{0,q}^{(d)}}
 =
 \frac{2q+1}{2q+d}
 \le1,
\]
so for even $n\ge4$ the largest value occurs at $q=2$, namely
\[
 \lambda_{0,2}^{(d)}=\frac{3}{d(d+2)}.
\]
Along the odd branch $n=2q+1$,
\[
 \frac{\lambda_{1,q+1}^{(d)}}{\lambda_{1,q}^{(d)}}
 =
 \frac{2q+3}{2q+2+d}
 \le1,
\]
so for odd $n\ge3$ the largest value occurs at $q=1$, namely
\[
 \lambda_{1,1}^{(d)}=\frac{3}{d(d+2)}.
\]
Consequently all components with $n\ge3$ have eigenvalue at most $3/[d(d+2)]$, with equality attained by $(m,j)=(1,1)$ in degree $3$ and by $(m,j)=(0,2)$ in degree $4$. Together with the trace-free $n=2$ value $2/[d(d+2)]$, this gives
\[
 Q_n(A)
 \leq
 \frac{3}{d(d+2)}\|A\|^2
\]
for every $n\ge3$, and for $n=2$ on the trace-free subspace. This proves Lemma~\ref{lem:tensor-estimate}.

\section{Grassmannian projection moments}\label{app:grassmannian}

In this appendix we give the computation used in
Lemma~\ref{lem:grassmannian_algebra}. The argument uses the irreducible
decomposition of symmetric two-tensors under the orthogonal group. The
remaining scalar is fixed by evaluating one trace-free test matrix. Related
formulae for Grassmannians may be found in
James--Constantine~\cite{JamesConstantine1974}.

Let \(G_{d,k}\) be the Grassmannian of \(k\)-dimensional linear subspaces
of \(\mathbb R^d\), equipped with its \(O(d)\)-invariant probability
measure \(\pi_{d,k}\). For \(E\in G_{d,k}\), let \(P_E\) denote the
orthogonal projection onto \(E\). For a symmetric \(d\times d\) matrix
\(M\), set
\[
Q(M):=\int_{G_{d,k}}\|P_E M(I-P_E)\|_{\mathrm{HS}}^2\,\dd\pi_{d,k}(E).
\]

\paragraph{Step 1: invariant reduction.}
For every \(R\in O(d)\), one has \(P_{RE}=RP_E R^{\mathsf T}\). Since
\(\pi_{d,k}\) is \(O(d)\)-invariant, it follows that
\[
Q(RMR^{\mathsf T})=Q(M).
\]
Thus \(Q\) is an \(O(d)\)-invariant quadratic form on the space
\(S_d\) of real symmetric \(d\times d\) matrices.

Moreover,
\[
Q(\alpha I)=0
\]
for every \(\alpha\in\mathbb R\),   In fact \(Q(M+\alpha I)=Q(M)\) for every \(\alpha\in\mathbb R\), since
\(P_E I(I-P_E)=0\).   Writing
\[
S_d^0:=\{M\in S_d:\operatorname{Tr}M=0\},
\]
we have the \(O(d)\)-orthogonal decomposition
\[
S_d=S_d^0\oplus \mathbb RI.
\]
Since \(Q\) vanishes on the scalar summand, it remains to understand
its restriction to \(S_d^0\).

The \(O(d)\)-representation \(S_d^0\) is irreducible. By Schur's lemma,
or equivalently by the uniqueness of the \(O(d)\)-invariant inner
product on this irreducible representation, the polarization of
\(Q|_{S_d^0}\) must be a scalar multiple of the Hilbert--Schmidt inner
product. Therefore there exists a constant \(C(d,k)\ge 0\) such that
\begin{equation}\label{eq:grassmannian-quadratic-form}
Q(M)=C(d,k)\left\|M-\frac{\operatorname{Tr}M}{d}I\right\|_{\mathrm{HS}}^2
\qquad\text{for all }M\in S_d.
\end{equation}
Since \(1\le k\le d-1\), the constant computed below is positive.
Consequently, if \(Q(M)=0\), then \(M\) is a scalar matrix.

\paragraph{Step 2: evaluation of the constant.}
We evaluate \eqref{eq:grassmannian-quadratic-form} at the normalized
trace-free matrix
\[
M_0:=\frac1{\sqrt2}\operatorname{diag}(1,-1,0,\ldots,0).
\]
Then \(\operatorname{Tr}M_0=0\) and \(\|M_0\|_{\mathrm{HS}}^2=1\), hence
\[
C(d,k)=Q(M_0).
\]

Write \(P_E=(p_{ab})_{a,b=1}^d\). Since \(P_E=P_E^{\mathsf T}=P_E^2\),
we have
\[
\|P_E M_0(I-P_E)\|_{\mathrm{HS}}^2
=\operatorname{Tr}\bigl(P_E M_0^2 P_E\bigr)
-\operatorname{Tr}\bigl(P_E M_0 P_E M_0 P_E\bigr).
\]
Using the fact that \(M_0\) has only the diagonal entries
\(1/\sqrt2\) and \(-1/\sqrt2\), this gives
\begin{equation}\label{eq:grassmannian-expansion}
\|P_E M_0(I-P_E)\|_{\mathrm{HS}}^2
=
\frac12\bigl(p_{11}-p_{11}^2+p_{22}-p_{22}^2\bigr)+p_{12}^2.
\end{equation}

We now compute the required moments. First, by invariance,
\[
\int_{G_{d,k}} P_E\,\dd\pi_{d,k}(E)=\frac{k}{d}I,
\]
and therefore
\begin{equation}\label{eq:grassmannian-first-moment}
\int_{G_{d,k}}p_{ii}\,\dd\pi_{d,k}(E)=\frac{k}{d}.
\end{equation}

For the second moments, \(O(d)\)-invariance implies that there are
constants \(A_{d,k}\) and \(B_{d,k}\) such that
\begin{equation}\label{eq:grassmannian-fourth-tensor}
\int_{G_{d,k}}p_{ab}p_{cd}\,\dd\pi_{d,k}(E)
=
A_{d,k}\delta_{ab}\delta_{cd}
+
B_{d,k}\bigl(\delta_{ac}\delta_{bd}+\delta_{ad}\delta_{bc}\bigr).
\end{equation}
The two constants are determined by the identities
\[
\operatorname{Tr}P_E=k,
\qquad
\operatorname{Tr}(P_E^2)=\operatorname{Tr}P_E=k.
\]
Indeed, summing \eqref{eq:grassmannian-fourth-tensor} over \(a,c\)
with \(b=a\), \(d=c\), we get
\[
k^2
=
\int(\operatorname{Tr}P_E)^2\,\dd\pi_{d,k}(E)
=
d^2A_{d,k}+2dB_{d,k}.
\]
Similarly, summing \eqref{eq:grassmannian-fourth-tensor} over \(a,b\)
with \(c=a\), \(d=b\), we get
\[
k
=
\int\operatorname{Tr}(P_E^2)\,\dd\pi_{d,k}(E)
=
dA_{d,k}+d(d+1)B_{d,k}.
\]
Solving these two equations yields
\[
B_{d,k}
=
\frac{k(d-k)}{d(d-1)(d+2)},
\qquad
A_{d,k}
=
\frac{k((d+1)k-2)}{d(d-1)(d+2)}.
\]
Consequently,
\begin{equation}\label{eq:grassmannian-second-moments}
\int_{G_{d,k}}p_{ii}^2\,\dd\pi_{d,k}(E)
=
A_{d,k}+2B_{d,k}
=
\frac{k(k+2)}{d(d+2)},
\end{equation}
and, for \(i\ne j\),
\begin{equation}\label{eq:grassmannian-offdiag-moment}
\int_{G_{d,k}}p_{ij}^2\,\dd\pi_{d,k}(E)
=
B_{d,k}
=
\frac{k(d-k)}{d(d-1)(d+2)}.
\end{equation}

Inserting
\eqref{eq:grassmannian-first-moment},
\eqref{eq:grassmannian-second-moments}, and
\eqref{eq:grassmannian-offdiag-moment}
into \eqref{eq:grassmannian-expansion}, we obtain
\begin{align*}
C(d,k)
&=
\frac12\cdot 2
\left(
\frac{k}{d}
-
\frac{k(k+2)}{d(d+2)}
\right)
+
\frac{k(d-k)}{d(d-1)(d+2)}
\\
&=
\frac{k(d-k)}{d(d+2)}
+
\frac{k(d-k)}{d(d-1)(d+2)}
\\
&=
\frac{k(d-k)}{(d-1)(d+2)}.
\end{align*}
Therefore
\begin{equation}\label{eq:grassmannian-coefficient}
\int_{G_{d,k}}\|P_E M(I-P_E)\|_{\mathrm{HS}}^2\,\dd\pi_{d,k}(E)
=
\frac{k(d-k)}{(d-1)(d+2)}
\left\|M-\frac{\operatorname{Tr}M}{d}I\right\|_{\mathrm{HS}}^2.
\end{equation}
This proves Lemma~\ref{lem:grassmannian_algebra}.
\addcontentsline{toc}{section}{References}

\providecommand{\href}[2]{#2}

\bigskip

\noindent\textbf{Declaration.} The author declares that he has no conflict of interest and that the manuscript has no associated data. 
During the preparation of this manuscript, the author used Kimi (Moonshot AI) and ChatGPT (OpenAI) for  language polishing and stylistic refinement. 
All mathematical statements remain the responsibility of the author.

\medskip

\noindent \textbf{Funding}:   This  work was supported by  the National key R \& D programs of China (2021YFA1000900, 2021YFA1002200),  National Natural Science Foundation of China  (12201596),  Shandong Provincial Natural Science Foundation (ZR2025QB05) and Taishan Scholars Program of Shandong Province (tsqn202408059).

\end{document}